\documentclass[12pt, a4paper]{amsart}

\usepackage[
  hmarginratio={1:1},     
  vmarginratio={1:1},     
  textwidth=148mm,        
  textheight=220mm,       
  marginparwidth=27mm,
  marginparsep=3mm
]{geometry}

\usepackage[utf8]{inputenc}
\usepackage{lmodern}
\usepackage{url}
\usepackage{amsmath, amsthm, amssymb, amscd, accents, bm}
\usepackage{mathtools}
\usepackage{mdwlist}
\usepackage{paralist}
\usepackage{graphicx}
\usepackage{epstopdf}
\usepackage{datetime}
\usepackage{afterpage}
\usepackage[dvipsnames]{xcolor}
\usepackage{hyperref}
\usepackage[ruled,vlined]{algorithm2e}
\usepackage{caption}
\usepackage{subcaption}
\mathtoolsset{showonlyrefs}
\usepackage[sort,nocompress]{cite}
\usepackage{marginnote}
\usepackage{comment}
\usepackage[foot]{amsaddr}
\usepackage{bm}
\usepackage{scalerel}
\usepackage[normalem]{ulem}

\usepackage{mathtools}
\usepackage{nicematrix}

\newcommand\rurl[1]{%
  \href{https://#1}{\nolinkurl{#1}}%
}

\mathtoolsset{showonlyrefs}

\theoremstyle{plain}
\newtheorem{definition}{Definition}[section]
\newtheorem{theorem}[definition]{Theorem}
\newtheorem{lemma}[definition]{Lemma}

\newtheorem{proposition}[definition]{Proposition}

\newtheorem*{theorem*}{Theorem}

\theoremstyle{definition}
\newtheorem{remark}[definition]{Remark}
\newtheorem*{remark*}{Remark}

\newtheorem*{remarks*}{Remarks}

\numberwithin{equation}{section}

\def\R{{\mathbb R}}
\def\T{{\mathbb T}}
\def\C{{\mathbb C}}

\newcommand{\Rd}{{\R^d}}
\newcommand{\Cd}{{\C^d}}

\newcommand{\diag}{\mathrm{diag}}

\newcommand{\ltd}{{L^2(\R^d)}}

\newcommand{\re}{{\mathrm{Re} \,}}
\newcommand{\im}{{\mathrm{Im} \,}}

\newcommand{\Vb}{{V}}
\newcommand{\AF}{\mathrm{A}}
\newcommand{\W}{\mathrm{W}}

\newcommand{\J}{{\mathcal{J}}}
\newcommand{\V}{{\mathcal{V}}}
\newcommand{\D}{\mathcal{D}}
\newcommand{\Ro}{\mathcal{R}}
\newcommand{\A}{\mathcal{A}}

\newcommand{\bJ}{\bm{\mathcal{J}}}
\newcommand{\bV}{\bm{\mathcal{V}}}
\newcommand{\bD}{\bm{\mathcal{D}}}
\newcommand{\bRo}{\bm{\mathcal{R}}}
\newcommand{\bA}{{\bm{\mathcal{A}}}}
\newcommand{\bB}{{\bm{\mathcal{B}}}}

\newcommand{\F}{{\mathcal{F}}}

\makeatletter
\newcommand\thankssymb[1]{\textsuperscript{\@fnsymbol{#1}}}

\newcommand{\Sp}{{\mathrm{Sp}(4d,\R)}}

\newcommand{\Spp}{{\mathrm{Sp}(2d,\R)}}
\newcommand{\Mpp}{{\mathrm{Mp}(2d,\R)}}

\makeatletter
\def\@makefnmark{%
  \leavevmode
  \raise.9ex\hbox{\fontsize\sf@size\z@\normalfont\tiny\@thefnmark}}

\makeatletter
\def\bign#1{\mathclose{\hbox{$\left#1\vbox to8.5\p@{}\right.\n@space$}}\mathopen{}}

\usepackage{xifthen}
\usepackage{eqparbox}
\usepackage{makebox}
 
\newcommand{\abs}[1]{\lvert #1\rvert}

\newcommand{\up}{uncertainty principle}

\newcommand{\tfr}{time-frequency representation}
\newcommand{\ft}{Fourier transform}
\newcommand{\stft}{short-time Fourier transform}

\newcommand{\tf}{time-frequency}

\newcommand{\fif}{if and only if}

\newcommand{\modsp}{modulation space}
\newcommand{\psdo}{pseudodifferential operator}

\newcommand{\beqa}{\begin{eqnarray*}}
\newcommand{\eeqa}{\end{eqnarray*}}

\DeclareMathOperator*{\supp}{supp}

\newcommand{\field}[1]{\mathbb{#1}}
\newcommand{\bR}{\field{R}}        
\newcommand{\bT}{\field{T}}        %
 \def\cA{\mathcal{A}}

\def\rd{\bR^d}

\def\rdd{{\bR^{2d}}}

\def\lrd{L^2(\rd)}
\def\lrdd{L^2(\rdd)}

\def\intrd{\int_{\rd}}

\def\<{\left<}
\def\>{\right>}

\def\inv{^{-1}}

\def\mv1{M_v^1}

\makeatletter

\protected\def\dwave{%
\leavevmode \bgroup
\markoverwith {%
\lower 3.5\p@ \hb@xt@ \z@{\sixly \char 58\hss}%
\lower 5\p@ \hbox {\sixly \char 58}%
}%
\ULon 
}

\makeatother

\begin{document}
\title[Uncertainty principle for metaplectic time-frequency representations]{Benedicks-type uncertainty principle for metaplectic time-frequency representations}
\address{\thankssymb{1}Faculty of Mathematics, University of Vienna, Oskar-Morgenstern-Platz 1, 1090 Vienna, Austria}

\author[\qquad \qquad \qquad \qquad \quad Karlheinz \ Gr\"ochenig]{Karlheinz Gr\"ochenig\thankssymb{1} 
}
\email{karlheinz.groechenig@univie.ac.at}

\author[Irina Shafkulovska\qquad\qquad\qquad ]{Irina Shafkulovska\thankssymb{1}
}
\email{irina.shafkulovska@univie.ac.at}
\thanks{I. Shafkulovska was funded in part by the Austrian Science Fund (FWF) [\href{https://doi.org/10.55776/P33217}{10.55776/P33217}]. 
 For open access purposes, the authors have applied a CC BY public copyright license to any
author-accepted manuscript version arising from this submission.}
\allowdisplaybreaks
\belowdisplayshortskip0pt

\date{\today}
\maketitle

\begin{abstract}
Metaplectic Wigner distributions
are  joint time-frequency representations that are parametrized by a
symplectic matrix  and  generalize  the short-time Fourier transform and the Wigner distribution.

We investigate the question  which metaplectic Wigner
distributions satisfy an uncertainty principle in the style of
Benedicks and Amrein-Berthier. That is, if the metaplectic Wigner distribution
is supported on a set of finite measure, must the functions then  be
zero? While this statement holds for the \stft , it is false for some  other
natural \tf\ representations.   We  provide a full
characterization of the class of metaplectic Wigner distributions
which exhibit an uncertainty principle of this type,   both for 
sesquilinear and quadratic versions. 

\noindent \textbf{Keywords.} Uncertainty principle, metaplectic Wigner distributions, support of time-frequency representations, symplectic group, metaplectic group, oscillator representation

\end{abstract}

\section{Introduction} 

The uncertainty principle refers to a huge collection of inequalities
that assert that a function and its Fourier transform cannot be small
simultaneously. Measuring the size of a function $f$ on the real line  by its   second moment $\Delta f =
\int _\bR  x^2 |f(x)|^2 \, dx$, one obtains the
classical \up, the Heisenberg-Pauli-Weyl inequality 
$$
\Delta f \, \Delta \hat{f} \geq \frac{1}{16 \pi ^2} \|f\|_2^4 \, . 
$$
Here  we use the normalization $\hat{f}(\omega ) = \F f(\omega) = \int_{\bR } f(x)
e^{-2\pi i x\omega } \, dx $ for the Fourier transform of $f$.

Measuring the size of $f$ by the Lebesgue measure of its  support,
one obtains a qualitative \up : 
\begin{center}
    \emph{If both $f$ and $\hat{f}$ have a
support of finite measure, \\ then $f$ is the zero function.}
\end{center}
This \up\ is due to Benedicks~\cite{Benedicks1985} and Amrein and
Berthier \cite{AmreinBerthier1977} and is  often referred to as
Benedicks's \up\  (because of an early preprint
of ~\cite{Benedicks1985}).  

In general, for every measure of size, one can obtain a version of the \up . 
For excellent surveys of such \up s  for the Fourier
pair $(f,\hat{f})$, we recommend the surveys~\cite{FollandSitaram1997,
  HavinJoericke1994, RicaudTorresani2014, BonamiDemange2006}. 

In quantum mechanics and signal processing, one often 
prefers to study 
joint \tf\ representations, such as  the \stft\ or  
the Wigner distribution, in place of the Fourier pair $(f,\hat{f})$.
The basic \tf\ representation is the   \stft\ of a function $f$ on
$\Rd$ with respect to a  window $g$, which is 
defined as
$$
\Vb_gf(x,\omega ) = \intrd f(t) \bar{g}(t-x) e^{-2\pi i \omega\cdot t}
\, dt \, .
$$
In the standard interpretation of the \stft ,  $\Vb_gf(x,\omega )$ is
the amplitude of the frequency $\omega$ of  $f$ locally at  time $x$. 
In this sense, 
the \stft\ 
carries information simultaneously about $f$
\emph{and} 
$\hat{f}$,  and therefore  one expects a corresponding   \up
. Indeed, the analog of Benedicks's \up\ holds for the \stft
~\cite{Jaming1998,Janssen1998_supp,
  GroechenigZimmermann2001,Wilczok2000}:
\begin{center}
\hspace{1 cm}    \emph{If $ \Vb_gf$ is supported on a set of finite measure, then $f\equiv 0$ or
  $g\equiv 0$.} \hfill   (*)
\end{center}

 The same result holds for close relatives of the \stft, such as  the ambiguity function  and the Wigner
distribution. The proof of (*) uses Benedicks's result for the Fourier
pair $(f,\hat{f})$. In fact, 
there is now an almost systematic technique to translate an  \up\
 from the Fourier pair $(f,\hat{f})$ to an \up\ of  a joint \tfr, e.g., $\Vb_g f$, 
 see the survey~\cite{Groechenig2003}.

 The \stft, the ambiguity function,  and the Wigner distribution are undoubtedly the most important
 and most frequently used time-frequency representations. 
 They are used for the definition of
 function spaces~\cite{Feichtinger2006,Groechenig2001},  in the calculus of \psdo s~\cite{Folland1989, Groechenig2001, Cordero_Rodino,BenyiOkodjou2020}, in many
 applications in signal processing~\cite{Woodward1964,
   gabor1946,Torresani1999,CarmonaEtAl1998, Flandrin1999},  and in the phase space formulation
 of quantum mechanics~\cite{Gosson2006,Gosson2011}. 
 Yet there is a whole zoo of equally
 interesting \tfr s available~\cite{HlawatschBoudreaux1992}. 
For instance, 
the popular $\tau $-Wigner distribution goes back at least to~\cite{janssentau}.
 Further extensions were
 studied as linear perturbations of the
 Wigner distribution~\cite{BayerPhD}, as canonical linear
 transforms~\cite{Zhang2019}, free metaplectic transformations \cite{Zhang2023},
 and as symplectic Wigner
 distributions~\cite{CorderoRodino2022}.
 
For this huge class of \tfr s, one expects many \up s to carry over
from the \stft   . This is the topic of our work. 

 To set up the results, we first recall the symplectic group and the
 associated metaplectic representation.
 Let  
\begin{equation}\label{eq:J_def}
    \J =
 \begin{pmatrix}
   0 & I \\ -I & 0
 \end{pmatrix}\in \R^{2d\times 2d}
\end{equation}  
be the standard skew-symmetric matrix in
$\mathrm{GL}(2d,\bR)$. Here $ 0 = 0_{d}$ and $I= I_d$
denote the zero and identity matrix in $\R^{d\times d}$
respectively~\footnote{We include the dimension in the index only when necessary, but will otherwise omit it.}.
 Then the symplectic group  $\mathrm{Sp}(2d,\bR )$  
 consists of all $2d\times 2d$-matrices $\A$ such that $\A^t
 \J \A = \J$. Associated to $\A \in \mathrm{Sp}(2d,\bR )$ is a
 unitary operator $\mu (\A )$ which intertwines the symmetric
 phase-space shifts and acts on $\lrd $ \cite{Groechenig2001, Folland1989, Gosson2011}.

 Now we double the dimension
 and assume that $\bA \in \mathrm{Sp}(4d,\bR ) $. Then $\mu (\bA )$ acts
 on $\lrdd $, and in particular on the elementary tensors $(f \otimes
 g)(x,y) = f(x)g(y)$ for $f,g\in \lrd $. Motivated by  a factorization of the \stft\ and
 the Wigner distribution~\cite[Lem.~3.1.2, 4.3.3]{Groechenig2001}, one can introduce a  \tfr\
 as
 \begin{equation}
   \label{eq:c1}
 \W_{\bA } (f,g) = \mu (\bA ) (f \otimes \bar{g}) \qquad f,g \in \lrd.
 \end{equation} 
Following Cordero and Rodino~\cite{CorderoRodino2022,
  CorderoRodino2023}, who studied \psdo s and \modsp s by means of
these transforms, we call this \tfr\ a metaplectic \tfr\ or an $\bA $-Wigner
distribution.
 All common \tfr s, in particular, the \stft , the ambiguity function, all $\tau $-Wigner
 distributions, the canonical linear transforms, and their
 extension in~\cite{BayerPhD}   can be written as a symplectic \tfr\ with a
 particular $4d \times 4d$ symplectic matrix $\bA $,  see Section~7. The question
 arises whether the usual \up s can be extended and  formulated for
a general symplectic \tfr. One would expect an affirmative answer,
 but after checking a few examples, one begins to understand that this question
 leads to an interesting problem, whose answer is by far not
 obvious. 

 \vspace{3mm}
 
 \textbf{Examples.} 
 {Since the intertwining operators $\mu(\A)\in \Spp$ are unique up to a multiplication with a global phase $\tau\in\T$ \cite{Gosson2011, Folland1989, Groechenig2001}, we omit the global phase in the examples below.} 
 
(i) Let $\bA = \bJ $ be the standard symplectic matrix in $\Sp$. Then $\mu (\bJ
)F = \hat{F}$ for $F\in \lrdd $ is just the Fourier transform. Evidently,
$\W_{\bJ} (f \otimes \overline{g}) = \hat{f} \otimes
\hat{\overline{g}}$ is  supported on a set of finite measure 
(or on a compact set) \fif\ both $\hat{f}$ and $\hat{g}$
have a support of finite measure (compact support) in $\rd$. So there is
no Benedicks-type \up .    

(ii)  Let $L\in \mathrm{GL}(2d,\bR )$ and $$\bA =
\begin{pmatrix}
  L & 0 \\ 0 & L^{-t}
\end{pmatrix} \in \mathrm{Sp}(4d,\bR ) \, .$$
 The corresponding  metaplectic operator is given by a unitary coordinate
 transform $\mu (\bA ) f(\lambda) = |\det L |^{-1/2} F(L \inv \lambda)$ for
 $\lambda\in \rdd $ and $F \in \lrdd $. Clearly, a coordinate transform
 changes the support and its measure, but the compactness and the
 finiteness of the measure are preserved. Thus, if we assume that the
 \tfr\ $\W_{\bA}(f \otimes \overline{g})$ has compact support,
 then  also $f \otimes {g}$ has compact support and every pair $(f,g)$
 of compactly supported functions in $\lrd $ yields a counter-example
 to the expected  version of Benedicks's theorem.

 (iii) Let $Q=Q^t\in \mathrm{GL}(2d,\bR )$ be symmetric and $$\bA = \begin{pmatrix}
  I & 0 \\ Q & I
\end{pmatrix} \in \mathrm{Sp}(4d,\bR ) \, .$$
Then the operator  $\mu (\bA ) F(\lambda) = e^{i\pi \lambda\cdot Q\lambda} F(\lambda)$ for $F\in \lrdd $ amounts to the
multiplication by a ``chirp''. Thus, $\supp\,\W_{\bA} (f \otimes \overline{g})
= \supp \, (f\otimes \overline{g})$, and again  every pair $(f,g)$
 of compactly supported functions in $\lrd $ yields a counter-example
 to the analogous version of Benedicks's theorem.

(iv)  By contrast, we can represent the \stft, the ambiguity function,
and  the Wigner  distribution in the form $\W_\bA (f,g)$ for some
 symplectic matrix $\bA \in \mathrm{Sp}(4d,\bR )$. In this case,
 if 
 $\supp\,\W_\bA (f,g)$ has finite measure, then $f\equiv 0$ or
 $g\equiv 0$ by the results in~\cite{Jaming1998,Janssen1998_supp,
  GroechenigZimmermann2001,Wilczok2000}. Thus a Benedicks-type \up\
holds. 

 These examples show that some metaplectic \tfr s obey an  \up\ in the
 style of Benedicks, whereas others violate this \up. 
 These initial examples look rather puzzling and have struck our
 curiosity and form the motivation for our work. 

 We investigate the Benedicks-type \up\ for general metaplectic \tfr s
 and will  provide a complete description of those 
 \tfr s that yield a Benedicks-type \up. As this description requires
 some structure theory of the symplectic group and many technical
 preliminaries, we formulate the  most important special case of
 Benedicks's  \up\ in order to offer a flavour of our general result.

 \begin{theorem} \label{thm:tmintro}
Let $\bA = \begin{psmallmatrix}  A & B \\ C & D\end{psmallmatrix}
\in \mathrm{Sp}(4d,\bR )$, and assume that $B$ is
invertible. Consider the 
symmetric matrix $P = B^{-1}A = \begin{psmallmatrix} P_{11} & P_{12} \\ P_{21} & P_{22}\end{psmallmatrix}$.
\begin{enumerate}[(i)] 
    \item If     $P$ is not block-diagonal, i.e., $P_{12} = P_{21}^t
      \neq 0$, then $\W_{\bA }(f,g)$
satisfies the
Benedicks-type \up:
\begin{center}
    If $\W_{\bA} (f,g)$ is supported on a set of finite measure, \\ then $f\equiv 0$ or $g\equiv
0$.
\end{center}
\item If $P$ is block-diagonal, i.e., $P_{12} = P_{21}^t = 0$, 
then there exist non-zero functions $f,g\in \lrd $, 
such that the support of $\W_{\bA}(f,g)$ is compact. 
\end{enumerate}
\end{theorem}
The matrices addressed in Theorem \ref{thm:tmintro} are called \emph{free symplectic matrices}.
In the full treatment of Benedicks's \up\ we will remove the
condition  $\det B \neq 0$ with a trick that might be useful in
other contexts as well. Multiplying $\bA $ with an orthogonal symplectic matrix of the form 
$$
\bRo _{\tau I} = \begin{pmatrix}
   a I & b I \\ -b I & a I
\end{pmatrix}\in \Sp \cap \mathrm{O}(4d,\R),
$$
with $\tau = a+i b\in\T$,  we can guarantee that the resulting
matrix $\bA \bRo _\tau $ is free for  all but
finitely many $\tau \in \T$, and 
Theorem~\ref{thm:tmintro} is applicable, see Lemma~\ref{lemma:tau-expansion}.  

The proof of Theorem \ref{thm:tmintro} is based on several  key elements.

(i) We use a version of the Iwasawa decomposition of
symplectic matrices that yields a factorization of a symplectic matrix
into three matrices of the type that were discussed in the examples
(i) --- (iii). In view of the (counter-)examples above, the main problem is
the maximal compact subgroup of $\mathrm{Sp}(4d,\bR )$, see Section \ref{sec:structure}. A main role in
the investigation will be played by a recent singular value
decomposition of orthogonal symplectic matrices. In this respect  we
would like to  acknowledge inspiration from  the note
of Jaming who implicitly used matrix factorizations
to prove uncertainty principles for the fractional Fourier transform~\cite{Jaming2022}.

(ii) We introduce  a new notion of a partial \stft\ by 
freezing  some variables and taking the \stft\ only over a subset of
variables. 

Our final result will be  an  algorithm for deciding 
which metaplectic Wigner distributions $W_\bA (f,g)$ obey the
Benedicks-type \up . Although our main theorem contains a complete
description of all metaplectic \tfr s for which 
Benedicks's \up\ holds, the deeper geometric meaning of the  conditions
still remains unclear to us. 

In the last part, we will investigate quadratic \tfr s of the form $\W_{\cA} (f,{f}) = \mu(\bA)(f \otimes \overline{f})$. It may seem that nothing new is gained by
studying quadratic \tfr s. However, the example of  the so-called
Rihaczek distribution defined by
\begin{equation}
    R(f,g)(x,\omega) = e^{-2\pi i x \cdot \omega} f(x) \overline{\hat{g}(\omega)}
\end{equation}
 points again to some mystery. Again $R(f,g)$ can be written as a
 symplectic \tfr.    
 For the sesquilinear Rihaczek
distribution $R(f,g)$, clearly  there are $f,g\in \lrd $, such
that $R(f,g)$ has compact support, and  Benedicks-type \up\ cannot be
true.  However, if the quadratic
expression $R(f,f)$ is supported on a set of finite measure, then
$f=0$. This  is just  the classical \up\ of Benedicks and Amrein-Berthier.  Thus there is
a subtle difference between quadratic and sesquilinear \tfr s.
In
Theorem~\ref{thm:auto_uncertainty}, we will characterize all 
quadratic symplectic \tfr\ $\W_{\bA}(f,f)$ that 
satisfy  Benedicks's \up. 

The paper is organized as follows:  in Section \ref{sec:SpMp} we
introduce the symplectic group and the metaplectic group. In Section
\ref{sec:UPFree}, we define  the {partial \stft} and  prove a first
Benedicks-type \up\ (Theorem \ref{thm:new_uncertainty_free}) that
implies   Theorem \ref{thm:tmintro} of the introduction. 
We then address the structure of the symplectic
group in Section \ref{sec:structure}. In Section \ref{sec:reductions},
we prove  the main result (Theorem \ref{thm:new_uncertainty}) and we formulate an algorithm for deciding 
which symplectic matrices obey the  Benedicks-type \up . 
 Section \ref{sec:auto} covers  \up s for quadratic \tfr s.
We illustrate the results by testing them on  the \stft, the ambiguity function and the ($\tau$-)Wigner distribution 
in Section \ref{sec:examples}.
\vspace{14mm}

\section{The symplectic group and the metaplectic group}\label{sec:SpMp}
\subsection{The symplectic group}
The symplectic group 
$\Spp$ 
consists of the matrices $\A\in \R^{2d\times 2d}$ 
which preserve
the standard skew-symmetric bilinear form, i.e., 
\begin{equation}\label{eq:Spp_def}
\Spp = \left\lbrace \A\in \R^{2d\times 2d}: \A^t \J \A = \J \right\rbrace,
\end{equation}
where $\J$ denotes the standard skew-symmetric matrix defined in \eqref{eq:J_def}. The symplectic group is a subgroup of the special linear group, and only when $d=1$ do these two groups coincide.
The simplest symplectic matrices are of the following type:\footnote{We use the convention $A+iB\in\C^{d\times d}$ with $A = \re(A+iB)$, $B = \im(A+iB)$.}
\begin{equation}\label{eq:Sp_standard_Ex}
\V_Q=\begin{pmatrix}
I & 0 \\
Q & I
\end{pmatrix},\quad
\D_L=\begin{pmatrix}
L & 0 \\
0 & L^{-T}
\end{pmatrix},\quad 
\Ro_{A+iB}= \begin{pmatrix}
    A & B \\ -B & A
\end{pmatrix}.
\end{equation}
The block structure of $\J$ provides four algebraic equations that characterize the blocks of symplectic matrices~\cite{Folland1989}. {With these equations}, one can easily verify that $\D_L$ are the only block-diagonal symplectic matrices, $\V_Q$ is symplectic if and only if $Q$ is symmetric, and $\Ro_{A+iB}$ is symplectic if and only if $U = A+iB$ is unitary on $\Cd$.

 The symplectic group is  a classical  Lie group. It possesses a 
 double cover, the so-called  \emph{metaplectic group} $\Mpp$ \cite{Knapp_LieBeyond, Hall2015}. 

\subsection{Metaplectic operators} 
Let $T_xf(t) = f(t-x)$ and $M_\omega f(t) = e^{2\pi i \omega \cdot
  t}$, $t,x,\omega \in \rd $, 
be the operators of translation and modulation. The 
symmetric time-frequency shifts 
$$
\rho(x,\omega) = T_{x/2} M_\omega T_{x/2}
$$ 
are a projective unitary representation of $\rdd$, as they satisfy
\begin{equation}\label{eq:rho_projective}
    \rho(\lambda+\gamma) = 
    e^{\pi i \lambda\cdot \J \gamma}
\rho(\lambda)\rho(\gamma),\qquad \lambda,\gamma\in\rdd.
\end{equation}
This relation extends to a unitary representation of the Heisenberg group \cite{Folland1989, Groechenig2001, Gosson2011}.
 From this perspective,
 the symplectic group is precisely the subgroup of $\mathrm{GL}(2d,\R)$ which preserve the phase factor $e^{\pi i \lambda \cdot \J \gamma}$ in \eqref{eq:rho_projective}.
 
 By results in representation theory of the Heisenberg group, specifically, the Stone-von-Neumann theorem \cite{Folland1989, Groechenig2001, Gosson2011}, 
 there exist unitary operators $\mu(\A)$  with
\begin{equation}\label{eq:Stone_von-Neumann} 
    \rho(\A\lambda) = \mu(\A)\rho(\lambda)\mu(\A)^{-1}, 
\end{equation}
that is, $\mu(\A)$ arises as an intertwining operator. These operators
are called \emph{metaplectic operators}. The operator is unique up to
a constant phase factor: it is obvious that any $\tau \mu(\A)$, $|\tau| =1$, 
also satisfies \eqref{eq:Stone_von-Neumann}. The mapping $\A \mapsto
\mu (\A)$ is a  projective, unitary  representation of the symplectic
group called the \emph{oscillator representation} \cite{Weil1964}. 

We already encountered the standard examples: the unitary dilations, the linear chirps, and the Fourier transform; 
\begin{equation}\label{eq:standard_metaplectic}
    \begin{split}
        \mu(\D_L) f(x) & = \abs{\det{L}}^{-1/2} f(L^{-1}x),\\
    \mu(\V_Q) f(x) & = e^{\pi i x^t  Qx}f(x), \\ \mu(\J) f(\omega)&= \hat{f}(\omega). 
    \end{split}
\end{equation}
 In this work, we will work with this particular choice of $\mu(V_Q),
 \mu(\D_L)$ and $\mu(\J)$, but the results are independent of the
 particular choice of the global phase.  In addition, one can show that
 with a suitable choice of the phase the metaplectic operators
 generate a group that is isomorphic to the  twofold cover of the symplectic group
 $\Spp$. This is the metaplectic group. The construction is rather subtle~\cite{Weil1964}, but is
 not needed in the sequel. In the following we identify the abstract
 twofold cover $\Mpp$ with the  group of operators generated by the
 intertwining operators $\mu (\A )$ \emph{with the correct phase factors}. 


With this designation, the projection 
\begin{equation*}
    \pi^{\mathrm{Mp}}:\Mpp\to\Spp
\end{equation*}
is a surjective group homomorphism, with kernel
$\mathrm{ker}(\pi^{\mathrm{Mp}}) =\left\lbrace \mathrm{id},
  -\mathrm{id} \right\rbrace$, and we often say that  the metaplectic
operator $\mu (\A ) \in \Mpp$ projects to the matrix $\A \in \Spp $.  

\subsection{Free symplectic matrices and quadratic Fourier transforms}
An important subset of $\Spp$  is the set
of \emph{free symplectic matrices}, i.e., the symplectic matrices with
an invertible upper-right $d\times d$ block. 
These possess a factorization into the elementary symplectic matrices~\eqref{eq:Sp_standard_Ex} as follows.     
\begin{proposition}\label{prop:free_Sp_decomposition} 
    Let $\A\in \Spp$ be a \emph{free symplectic matrix}, i.e., $\A$  has a block structure
    \begin{equation}
        \A = \begin{pmatrix}
            A & B \\ C & D
        \end{pmatrix}
    \end{equation}
    with $A, B,C,D \in \R^{d\times d}$ and invertible
    $B$. 
    Then 
    $$
    \A = \V_{D B^{-1}} \D_{B} \J  \V_{B^{-1} A}.
    $$
\end{proposition}
The proof is a short computation and can be found in \cite[Prop.~2.39.]{Gosson2006}.

As a consequence, metaplectic operators
corresponding to a free symplectic matrix  can  then be factored into the
elementary operators (Fourier transform, linear coordinate change, and
multiplication by a chirp) up to a phase factor. The composition of
these elementary operators  yields
integral operators with a Gaussian
kernel~\cite{terMorscheOonincx2002O}, which are sometimes   called
\emph{quadratic Fourier transforms} \cite{ Gosson2011}. Although we do
not need this fact, we mention that every symplectic matrix is the
product of two free symplectic matrices. Thus the matrices $\V_{Q},
\D_{L}, $ and $ \J$ generate the symplectic group. 


\textbf{Notation.}    We distinguish matrices in $\Spp$ and $\Sp$ by
the following notation where the basic matrices on $\R ^{4d} $ are in
bold face: 
\begin{align}
\mathcal{A}, \quad \V_{Q}, \D_{L},\Ro_{U}  &\in \Spp,\ Q,L\in\R^{d\times d}, U\in \C^{d\times d}, \\
 \bm{\mathcal{A}},\quad   \bV_{Q}, \bD_{L},\bRo_{U}  &\in \Sp,\ Q,L\in\R^{2d\times 2d}, U\in \C^{2d\times 2d}.
\end{align}
The blocks of a symplectic matrix are in general denoted with
$A,B,C,D$. For special matrices we use  the following convention: $L$
is used for invertible matrices, $W,V$ denote 
orthogonal matrices, $P,Q$ denote symmetric matrices, $U$ is always a 
unitary matrix,  and $\Sigma$ stands for a unitary diagonal
matrix. Note that only $U$ and $\Sigma$ are \emph{complex-valued}. For
$A, B\in\C^{d\times d}$ we denote with $\diag(A, B) = \begin{psmallmatrix}
    A & 0 \\ 0& B
\end{psmallmatrix}$  block-diagonal matrices with $d\times d$ blocks.

\begin{remark}\label{rem:embedding}
    One can easily verify that $\Spp\times \Spp$ can be embedded in
    $\Sp$ by embedding the generators $\D _L, \V_Q$ and $\Ro _U$ 
    in the following way:
\begin{equation}
    \begin{split}
        (\D_{L_1},\D_{L_2}) &\mapsto \bD_{\diag(L_1,L_2)},\quad L_1, L_2 \in\mathrm{GL}(d,\R),\\
        (\V_{Q_1},\V_{Q_2}) &\mapsto \bV_{\diag(Q_1,Q_2)}, \quad Q_1=Q_1^t, Q_2=Q_2^t \in\R^{d\times d},\\
        (\Ro_{U_1},\Ro_{U_2}) &\mapsto \bRo_{\diag(U_1,U_2)}, \quad U_1, U_2\in\mathrm{U}(d,\C).
    \end{split}
\end{equation}
    On the operator level, the action on the elementary tensors is given by
  \begin{align}  
        \mu(\bD_{\diag(L_1,L_2)}) (f\otimes g) &= \mu(\D_{L_1})
     f\otimes   \mu(\D_{L_2}) g, \notag \\ \label{ch1} 
        \mu(\bV_{\diag(Q_1,Q_2)})(f\otimes g) & = \mu(\V_{Q_1}) f\otimes \mu(\V_{Q_2}) g,\\
        \mu(\bRo_{\diag(U_1,U_2)})(f\otimes g) & =  \mu(\Ro_{U_1})
    f\otimes   \mu(\Ro_{U_2}) g. \notag 
\end{align}
\end{remark}
We will frequently use  the following formulas for the elementary
symplectic matrices  $\V_Q $,  $\D_L $, $\Ro_{U}$.
\begin{lemma}\label{rem:Sp_rechenregel}
Let  $Q, Q_1, Q_2\in\R^{d\times d}$ be symmetric, $L,L_1,L_2\in
\mathrm{GL}(d,\R)$, $U,U_1,U_2 \in \mathrm{U}(d,\C)$. Then: 
\begin{enumerate}[(i)]
        \item If $U$ is real-valued, i.e., $U\in \textrm{O}(d,\R)$,
            then  $\Ro_{U} = \D_U$. 
        \item Interaction with the \ft : 
          $\J \D_L = \D_{L^{-t}}\J$. 
        \item For $W\in \mathrm{O}(d,\R)$ we have
          \begin{equation}
            \D_W \V_Q \D_{W^t} = \V _{WQW^t}\, .
          \end{equation}
        \item Homomorphisms: 
                  \begin{equation}
            \V_{Q_1}\V_{Q_2} = \V_{Q_1+Q_2},\quad 
            \D_{L_1}\D_{L_2} = \D_{L_1L_2},\quad 
            \Ro_{U_1}\Ro_{U_2} = \Ro_{U_1U_2}.
        \end{equation}
    \end{enumerate}
  \end{lemma}                                

\begin{proof}
These properties follow immediately from the definition of the
elementary symplectic matrices~\eqref{eq:Sp_standard_Ex}.  
\end{proof}


\section{Uncertainty principles for free symplectic
  matrices}\label{sec:UPFree} In this section, we will prove a
Benedicks-type uncertainty principle for the $\bA$-Wigner distribution of
a free
symplectic matrix $\bA\in\Sp$. The general case will be reduced to
this case in Section~5.

Given a function $f\in\ltd$ and writing $x=(x_1,x_2)$ for $x_1\in\R^k$ and $x_2\in\R^{d-k}$ 
for some $1\leq k< d$, we denote with $f_{x_2}$ the restriction 
\begin{equation}
    f_{x_2}:\R^k\to\C,\quad x_1\mapsto f(x_1,x_2).
\end{equation}
By Fubini's theorem, for almost all $x_2\in \R^{d-k}$, the function
$f_{x_2}$ is square-integrable. This fact  allows us to define the partial \stft\ in $k$ variables almost everywhere.

\begin{definition} \label{def:partial_stft}
    Let  $f,g\in\ltd,\, g\not\equiv 0$, and let $k\in\{1,\dots, d\}$. We define the partial short-time Fourier transform in the first $k$ variables of the function $f$ with respect to $g$ as the 
    function
    
    \begin{equation}\label{eq:def_partial_STFT}
        \begin{split}
            \Vb^k_{g}f(x_1,x_2, \omega_1,\omega_2) & = \int_{\R^k} f(t,x_2) \overline{g(t-x_1, -\omega_2)} e^{-2\pi i t\cdot \omega_1 }\, dt \\
            & = \langle 
            f_{x_2},
            M_{\omega_1} T_{x_1} 
            g_{-\omega_2}
            \rangle \\ 
            & = \Vb_{
            g_{-\omega_2}
            } 
            f_{x_2}
            (x_1,\omega_1),
        \end{split}
    \end{equation}
    for $(x_1,x_2,\omega_1,\omega_2) \in \R^{k}\times \R^{d-k} \times \R^{k}\times \R^{d-k}$.

\end{definition}
If $k=d$, $\Vb^d_g = \Vb_g$ is the classical \stft. Otherwise, the integral in \eqref{eq:def_partial_STFT} is well-defined almost everywhere because $f_{x_2}$ and $g_{\omega_2}$ are square-integrable for almost all $x_2,\omega_2\in \R^{d-k}$.
For such $x_2,\omega _2$,  $\Vb _g^kf$  is continuous in $x_1,\omega_1\in\R^k$.


Like the \stft, the partial \stft\ can be interpreted as a metaplectic
\tf\ representation. To see this, we note that 
$(x_1,x_2,t,\omega _2) \mapsto (t,x_2,t-x_1,-\omega _2)$ is a linear
mapping $L_k^{-1} \in
\mathrm{GL}(2d,\R)$ on $\rdd $, and the
partial \ft\ is the metaplectic operator that 
projects to the matrix 
\begin{equation}
\bJ_k = \bRo_{\diag(I_k, I_{d-k}, i I_k, I_{d-k}}) \in\Sp. 
\end{equation} 
Therefore the partial \stft\ $V^k_g f$ is the metaplectic \tf\
distribution $\W_{\bJ_k \bD_{L_k}}(f,g)$.  

The following lemma is a Benedicks-type \up\  for the partial \stft ,
and is  a first step towards a general result. 

\begin{lemma}\label{lem:supp_partial_STFT}     Let $f,g\in\ltd$ and let $k\in\{1,\dots, d\}$.
If $\Vb^k_g f$ is supported on a set of finite measure, then $f\equiv 0$ or $g\equiv 0$.
\end{lemma}
\begin{proof}
 If $k = d$, then this is Benedicks's uncertainty principle for the
 \stft \cite{Janssen1998_supp, Jaming1998,
   GroechenigZimmermann2001,Wilczok2000}. Therefore, we may assume
 that  $1\leq k<d$.
 
 By Fubini's theorem, 
there exists a set $ N_1 \subseteq\rdd$ of measure zero such that for all $(x_1, x_2,\omega_1,\omega_2)\in \rdd \setminus N_1$, the functions $f_{x_2}, g_{\omega_2}$
are square-integrable and the representation \eqref{eq:def_partial_STFT} holds pointwise on $\rdd\setminus N_1$.

Let $\mathcal{S}\subseteq \rdd\setminus N_1$ be the support of $\Vb^k_g f$, where we have already omitted the null set of restrictions $f_{x_2},\, g_{\omega_2}$, $x_2,\omega_2\in\R^{d-k}$ that are meaningless or not in $L^2(\R^k)$. By 
Fubini's theorem, 
there exists a further set $N_2\subseteq \R^{d-k} \times \R^{d-k}$ of measure zero 
such that for fixed $(x_2,\omega_2)\in (\R^{d-k} \times \R^{d-k})\setminus N_2$ every section
$$
\mathcal{S}_{x_2,\omega_2} \coloneqq \left\lbrace (x_1,\omega_1) \in \R^{k}\times\R^{k} :\  (x_1, x_2,\omega_1,\omega_2) \in\mathcal{S}
\right\rbrace
$$
\begin{enumerate}[(i)]
    \item is measurable and 
has finite measure in $\R^k\times \R^k$,
\item the measure of $\mathcal{S}$ is given by
\begin{align}\label{eq:S_Cavalieri}
    |\mathcal{S}| = \int_{(\R^{d-k}\times \R^{d-k}) \setminus N_2}
  |\mathcal{S}_{x_2,\omega_2}| \, dx_2 d\omega_2,
\end{align}
\item $\mathcal{S}_{x_2,\omega_2} = \supp \Vb_{g_{-\omega_2}} f_{x_2}$.
\end{enumerate} 
Since $\mathcal{S}_{x_2,\omega_2}$ has finite measure in $\R^k\times \R^k$, Benedicks's uncertainty principle for the \stft \cite{Janssen1998_supp, Jaming1998, GroechenigZimmermann2001, Wilczok2000} implies that $f_{ x_2} \equiv 0 $ or $g_{ \omega_2} \equiv 0 $, and as a consequence,  $|\mathcal{S}_{x_2,\omega_2}| =0$.

   Going back to \eqref{eq:S_Cavalieri}, this implies that $|\mathcal{S}| =0$, hence $ \Vb^k_g f \equiv 0$. 
   As already noted in Definition \ref{def:partial_stft}, $\Vb^k_g f = \mu(\bA _k) (f\otimes \overline{g})$ for a suitable symplectic matrix $\bA _k\in\Sp$.
   Since $\mu(\bA _k)$ is unitary, it follows that $f\otimes \overline{g} \equiv 0$, or equivalently, $f\equiv 0$ or $g \equiv 0$.
   This proves the claim.
\end{proof}

As noted in the introduction, the operators $\mu (\D _L)$ and $\mu (\V
_Q)$  are not relevant for the validity of a Benedicks-type \up
. Since we will use this fact frequently, we state it formally. 

\begin{lemma}\label{lem:reduce_ortho}
    Let $\bA = \bV_Q \bD_L \bB\in\Sp$ and 
   $f,g\in \ltd$ non-zero. Then 
    \begin{equation}\label{eq:red_ortho}
        |\supp\, \W_{\bA}(f,g)| <\infty
    \end{equation}
    if and only if 
    $$
    |\supp\, \W_{\bB}(f,g)| <\infty.
    $$
\end{lemma}
\begin{proof}
By \eqref{eq:standard_metaplectic}, for 
 $f,g\in\ltd$, $\lambda\in\rdd$, we have
\begin{equation}
         \begin{split}
             |\W_\bA (f,g)(\lambda)|
             & = |\mu(\bV_Q)\mu(\bD_L) \mu(\bB)(f \otimes \overline{g})(\lambda)| \\
             & = |\det L|^{-1/2} |e^{\pi i \lambda\cdot Q \lambda}\mu(\bB)(f \otimes \overline{g})(L^{-1}\lambda)|.
         \end{split}
     \end{equation}
     Therefore,
    \begin{equation}
        \supp\, \W_{\bA}(f,g) =  L\cdot \supp\,\W_{\bB}(f,g).
    \end{equation}
    Clearly  $\supp\,\W_{\bA }(f,g)$ has finite measure,  
    if and only if 
    $\supp\,\W_{\bB}(f,g)$ has finite measure.
\end{proof}

We now prove an uncertainty principle for the $\bA$-Wigner
distribution of a special class of  free symplectic matrices. 
\begin{theorem}\label{thm:new_uncertainty_free} 
Let $\bA = \bJ \bV_P = \begin{psmallmatrix}
    P & I \\ -I & 0
\end{psmallmatrix}\in\Sp$ be a free symplectic matrix with 
$$
P = \begin{pmatrix}
    P_{11} & P_{12} \\ P_{21} & P_{22}
\end{pmatrix} \in\R^{2d\times 2d}, 
$$
such that the $d\times d$ blocks on the side diagonal $P_{12}=P_{21}^t \neq 0$ are not zero.  Then Benedicks's uncertainty principle holds:
    \begin{center}
    If $ \mathrm{W}_\bA (f,g)$ is supported on a set of finite measure for some $f,g\in\ltd$, \\
    then $f\equiv 0$ or
  $g\equiv 0$.
\end{center}
\end{theorem}

\begin{proof}

Let $f,g\in\ltd$ be two functions such that  $\W_{\bA}(f,g)$ is supported on a set of finite measure. We want to show that $f\equiv 0$ or $g\equiv 0$. 
We begin with a reduction to metaplectic operators with a simple action on elementary tensors. 


    Since
    $P_{12} = P_{21}^t\neq 0$,  it 
    has a singular value decomposition
    \begin{equation}
    P_{12} = W_1\cdot \Gamma \cdot W_2^t,
    \end{equation} 
    where $W_1, W_2\in \mathrm{O}(d,\R)$ and $\Gamma$ is a diagonal
    matrix  with non-negative entries, 
    i.e.,
    $$
    \Gamma = \begin{pmatrix}
        \Gamma_1 & { 0}_{k\times k} \\ 
        { 0}_{d-k \times k} & { 0}_{d-k \times d-k}
    \end{pmatrix},\quad \Gamma_1 \in \mathrm{GL}(k,\R),\quad k = \mathrm{rank}(P_{12})\geq 1.
    $$
    We write
    \begin{align*}
      P &= \begin{pmatrix}
            0 & P_{12} \\ P_{12}^t & 0  
        \end{pmatrix} + \begin{pmatrix} P_{11} &0\\ 0 & P_{22}\end{pmatrix}\\
&= \begin{pmatrix} W_1 &0\\0&W_2 \end{pmatrix}\, \begin{pmatrix} 0&\Gamma\\\Gamma&0 \end{pmatrix}\,
\begin{pmatrix} W_1^t &0\\0&W_2^t \end{pmatrix}+\begin{pmatrix} P_{11} &0\\ 0 & P_{22}\end{pmatrix}.
    \end{align*}
Using Lemma~\ref{rem:Sp_rechenregel}, 
we can write $\bA$ as
    \begin{equation}
    \begin{split}
        \bA& =\bJ \bV_P  
        = \bJ \bD_{\diag(W_1,W_2)} \bV_{\begin{psmallmatrix}
            0 & \Gamma \\ \Gamma & 0  
        \end{psmallmatrix}} \bD_{\diag(W_1,W_2)^t} \bV_{\diag(P_{11}, P_{22})} \\[1ex]
        & = \bD_{\diag(W_1,W_2)} \bJ \bV_{\begin{psmallmatrix}
            0 & \Gamma \\ \Gamma & 0  
        \end{psmallmatrix}} \bD_{\diag(W_1^t,W_2^t)} \bV_{\diag(P_{11}, P_{22})}.
    \end{split}
    \end{equation}
We note that $\mu(D_{W_1^t}\V_{P_{11}}), \mu(D_{W_2^t}\V_{-P_{22}})$
belong to $\Mpp$ and set   $f_1\coloneqq  \mu(\D_{W_1^t}\V_{P_{11}})f$, $ g_1\coloneqq \mu(D_{W_2^t}\V_{-P_{22}}) g\in\ltd$. 
Using Remark \ref{rem:embedding} and $\mu (\J ) = \F$, we can rewrite
the absolute value of $\W_{\bA}(f,g)$ as
\begin{align}
    |\W_{\bA}(f,g)| = 
    \big|\mu(\bD_{\diag(W_1,W_2)}) \F \mu\big(\bV_{\begin{psmallmatrix}
            0 & \Gamma \\ \Gamma & 0  
        \end{psmallmatrix}}\big) (f_1 \otimes\overline{g_1})\big|.
\end{align}
By Lemma~\ref{lem:reduce_ortho}, 
      $\F \mu(\bV_{\begin{psmallmatrix}
            0 & \Gamma \\ \Gamma & 0  
        \end{psmallmatrix}}) (f_1 \otimes\overline{g_1})$ 
        is also  supported on a set of finite measure.
        
By  the definition of $\V _P$ in \eqref{eq:standard_metaplectic}, 
\begin{equation}
    \mu(\bV_{\begin{psmallmatrix}
            0 & \Gamma \\ \Gamma & 0  
        \end{psmallmatrix}}) (f_1 \otimes\overline{g_1})(x,y) =
      e^{2\pi i \Gamma x\cdot  y} f_1(x) \overline{g_1}(y).
\end{equation}
This function is in $L^2(\rdd )$, and, with a suitable interpretation the
integrals,  its Fourier transform is 
    \begin{align}
        & \phantom{=} \F \mu(\bV_{\begin{psmallmatrix}
            0 & \Gamma \\ \Gamma & 0  
        \end{psmallmatrix}}) (f_1 \otimes\overline{g_1})(\omega, \xi) 
   \notag     \\
        & = \int_\Rd f_1(x) e^{-2\pi i \omega\cdot x} \, \int_\Rd
          \overline{g_1}(y)  e^{2\pi i \Gamma x\cdot  y} \, e^{-2\pi i
          \xi \cdot y} \, d y \, d x \notag \\
        & = 
        \int_\Rd f_1(x) e^{-2\pi i \omega \cdot x}  \, \overline{\hat{g_1}} (\Gamma x -\xi)  \ dx .
          \label{eq:free_integral_alt}
    \end{align}

If $\Gamma$ were invertible, then we could identify the above expression, after a coordinate change, as a \stft\ with a support of finite measure, and we could finalize the proof.
However, in general, the off-diagonal blocks $P_{12}$, $P_{21} = P_{12}^t$, and thus $\Gamma$, are \emph{not} invertible. 
This is the reason why we have to use the partial \stft\ in \eqref{eq:def_partial_STFT}. 
To {process} the integral in \eqref{eq:free_integral_alt} further, we evaluate $\W_\bA(f,g)$ at the points
\begin{equation}
    \big(\diag(\Gamma_1, I) \omega,  \xi \big)=(\Gamma_1 \omega_1,
    \omega_2,  \xi)\in \R ^k \times \R ^{d-k}\times \R ^d 
\end{equation}
instead of at $(\omega,\xi)$. With this coordinate change, the final expression in \eqref{eq:free_integral_alt} becomes
\footnote{If $k =d$, we set the inner integral over $\R^{d-k}$ to be the point evaluation of the integrand.}
\begin{equation} \label{eq:free_integral_f3_alt}
\begin{split}
        &\phantom{=}
        \F\mu(\bV_{\begin{psmallmatrix}
            0 & \Gamma \\ \Gamma & 0
        \end{psmallmatrix}})(f_1\otimes \overline{\hat{g_1}})\big( \diag(\Gamma_1, I) \omega, \xi \big) \\
        &=
        \int_\Rd f_1(x) e^{-2\pi i \omega \cdot \diag(\Gamma_1, I)x}\, \overline{\hat{g_1}} (\Gamma x - \xi)  \ dx \\
        &=
        \int_{\R^k} 
        \overline{\hat{g_1}} (\Gamma_1 x_1 -\xi_1, -\xi_2)
        e^{-2\pi i \omega_1\cdot \Gamma_1 x_1} 
        \int_{\R^{d-k}} f_1(x_1,x_{2}) e^{-2\pi i \omega_2\cdot x_2 } 
        \, d x_2 \ 
        d x_1.
        \end{split}
 \end{equation}
 We set $g_2= \widehat{g_1}$ and 
 \begin{equation}
     \begin{split}
         f_2(\Gamma_1 x_1,\omega_2) & \coloneqq \int_{\R^{d-k}} f_1(x_1,x_{2}) e^{-2\pi i \omega_2\cdot x_2} \, d x_2,\\
     \end{split}
 \end{equation}
Both $f_2$ and $g_2$ are (rescaled) partial Fourier transforms of $f_1$ and $g_1$, therefore $f_2,g_2\in\ltd$
 (note again  that the integral defining $f_2$ has to be interpreted
 in the sense of Plancherel's theorem). 
With this definition and the substitution $t = \Gamma_1 \eta_1$, \eqref{eq:free_integral_f3_alt} can be turned into the final form
        \begin{equation} \label{eq:free_integral_final_alt}
        \begin{split}
         &\phantom{=}
        \F\mu(\bV_{\begin{psmallmatrix}
            0 & \Gamma \\ \Gamma & 0
        \end{psmallmatrix}})(f_1\otimes \overline{\hat{g_1}})\big( \diag(\Gamma_1, I) \omega,\xi \big) \\
        &=
        \int_{\R^k}  f_2(\Gamma_1 x_1, \omega_2)  \overline{g_2} (\Gamma_1 x_1 -\xi_1, -\xi_2)  e^{-2\pi i \omega_1\cdot \Gamma_1 x_1} \, d x_1 \\
        & \overset{t = \Gamma_1 x_1}{=} |\det \Gamma_1|^{-1} 
        \int_{\R^k}  f_2(t, \omega_2)  \overline{g_2} (t -\xi_1, -\xi_2)  e^{-2\pi i \omega_1\cdot t} \, d t
        \\
        & = |\det \Gamma_1|^{-1} \Vb_{g_2}^k f_2 (\xi_1, \omega_2, \omega_1 , \xi_2).
        \end{split}
        \end{equation}
        Since $\W_\bA(f,g)$ has a support of finite measure, the transformed expression \eqref{eq:free_integral_final_alt} also has a support of finite measure.
By Lemma \ref{lem:supp_partial_STFT}, this implies that $f_2 \equiv 0$ or $g_2 \equiv 0$. 
Since $f_2$ and $g_2$ are the image of $f$ and $g$ under unitary operators, it follows that $f\equiv 0$ or $g \equiv 0$.
This concludes the proof.
\end{proof}

\begin{proof}[Proof of Theorem~ \ref{thm:tmintro}] 
The theorem formulated in the introduction is a direct consequence of   Theorem \ref{thm:new_uncertainty_free}.  Let $\bA = \begin{psmallmatrix}
        A & B \\ C & D
    \end{psmallmatrix}\in\Sp$
    be a free symplectic matrix. We set $ 
        Q = D B^{-1}$ and $ P = B^{-1}A.$
     By Proposition \ref{prop:free_Sp_decomposition},  $\bA $ factors as 
     \begin{equation}
         \bA  = \bV_{Q} \bD_{B} \bJ  \bV_{P}.
     \end{equation} 
    By Lemma~\ref{lem:reduce_ortho}, 
    $\W_\bA(f,g)$ is supported on a set of finite measure for
    some $f,g\in\ltd$ if and only if $\W_{\bJ \bV_P}(f,g)$ is
    supported on a set of finite measure.
    
    \textit{(i)}  Assume $P$ is not block-diagonal.  Then  Theorem
    \ref{thm:new_uncertainty_free}  implies that $f\equiv 0$ or $g\equiv 0$, as claimed.

    \textit{(ii)} If $P = \begin{psmallmatrix}
        P_{11} & 0 \\ 0 & P_{22}
    \end{psmallmatrix}$
    is block-diagonal, then by Remark \ref{rem:embedding}
    \begin{equation}
        \begin{split}
            |\W_{\bJ \bV_P}(f,g)(\lambda)| & = | \mu (\bJ ) \big((\mu(\V_{P_{11}})f\otimes \overline{\mu(\V_{-P_{22}})g}\big)(\lambda)| \\
            & = \left|\big(\mu(\J\V_{P_{11}})f\otimes \overline{\mu(\J^t\V_{-P_{22}})g}\big)(\lambda)\right|. 
        \end{split}
    \end{equation}
Now choose  non-zero $f_0,g_0\in\ltd$, 
with compact support and set 
    \begin{equation}
        f =\mu(\J\V_{P_{11}})^{-1}f_0, \qquad g
        =\mu(\J^{t}\V_{-P_{22}})^{-1}g_0 \, .
    \end{equation}
     Then $W_{\bJ \bV_P}(f,g)$ satisfies  
    \begin{equation}
        |\W_{\bJ \bV_P}(f,g)(\lambda)| = |f_0\otimes
        \overline{g_0}(\lambda)|\, ,
    \end{equation}
    and has compact support. Consequently, $\W _{\bA } (f,g)$ has also
    compact support. This concludes the proof.
\end{proof}

\section{The structure of the symplectic group}\label{sec:structure}
Before treating
uncertainty principles for arbitrary symplectic matrices, 
we need a deeper 
grasp 
of the structure of the symplectic group and its factorizations. Our
main tool is a decomposition of  symplectic matrices that resembles
the Iwazawa decomposition and is called the pre-Iwazawa
decomposition~\cite{Gosson2006}.

\begin{lemma}
  Every symplectic matrix $\A    \in \Sp $ can be factored as
  \begin{equation}\label{eq:iwasawa}
    \A = \V_Q \D_L \Ro_U = \begin{pmatrix}
        I & 0 \\
        Q & I
    \end{pmatrix}\cdot \begin{pmatrix}
      L & 0 \\ 0 & L^{-t}  \end{pmatrix}
      \cdot \begin{pmatrix}
        A & B \\
        -B & A
    \end{pmatrix}
\end{equation}
with $L\in\mathrm{GL}(d,\R)$,  $Q=Q^t\in\R^{d\times d}$, and $U=
A+iB\in\mathrm{U}(d,\C)$. This decomposition is unique if $L$ is
symmetric. 
\end{lemma}
\begin{proof}
  The decomposition can be derived easily from the Iwazawa
  decomposition of $\Sp$~\cite{Iwasawa1949, Knapp_LieBeyond}. A nice
  and direct proof was given by deGosson~\cite[Sec.~2.2.2]{Gosson2006}. It has the
  additional advantage that it does not require any knowledge of Lie
  group theory and provides explicit formulas for $Q,L,A,B$. 
\end{proof}

The advantage of \eqref{eq:iwasawa} is the fact that 
the operators $\mu(\V_Q),\,\mu(\D_L)$ do not change the nature of the
support of a function,  so that  a support of
finite measure (or compact support) is  preserved, as observed in
Lemma~\ref{lem:reduce_ortho}.  
Thus the  factor for the validity of an uncertainty principle is played by the symplectic rotations $\Ro_U\in \Spp$.
Since we have already established an uncertainty principle for free
symplectic matrices, we aim to reduce the general case to this special case.

We make use of the following result
\cite[Thm.~5.1.]{FuerRzeszotnik2018}, which is a  version of a joint singular value decomposition over $\R$ of the real and imaginary part of $U\in\mathrm{U}(d,\C)$. 
\begin{lemma}\label{lem:unitary_decomp}
 For every $U =A+iB\in\mathrm{U}(d,\C)$ there exist orthogonal matrices $W, V\in\mathrm{O}(d,\R)$ and a diagonal unitary matrix $\Sigma = \Sigma_r+i\Sigma_i \in\mathrm{U}(d,\C)$ such that
    $$
    U = W \Sigma V^t, \quad \text{i.e., }\quad A = W\Sigma_r V^t,\ B = W \Sigma_i V^t.
    $$
\end{lemma}

To move from free  symplectic matrices to general symplectic matrices, we use the following trick.
\begin{lemma} 
\label{lemma:tau-expansion} 
    For every $U = A+iB\in \mathrm{U}(d,\C)$  there exists a $\tau\in\T$ such that the orthogonal symplectic  matrix 
    $$
    \Ro_{\tau U} = \Ro_{\tau I} \Ro_{U} = \Ro_{U} \Ro_{\tau I}
    $$ 
    is \emph{free} and  $\im(\tau U)$ is invertible. 
    Furthermore, for a fixed $U$,   $\im(\tau U)$ fails to be invertible 
    for at most $2d$ choices of $\tau\in\T$.
\end{lemma}
\begin{proof}
    By Lemma \ref{lem:unitary_decomp}, $U$ factors as $U=W\Sigma V^t$
    with 
    $W,V\in\mathrm{O}(d,\R)$ and a diagonal unitary matrix $\Sigma =
    \diag( \Sigma _{kk})_{1\leq k\leq d} \in \mathrm{U}(d,\C)$
    with $|\Sigma _{kk}|=1$. 
Then     for all $\tau\in\T$,
    \begin{equation*}
        \begin{split}
            \im(\tau U) &
            = W\cdot \im(\tau \Sigma) \cdot V^t.
        \end{split}
    \end{equation*} 
    This shows that the 
    diagonal matrix $\im(\tau \Sigma)$ has zeros on the diagonal if
    and only if $\tau = \pm \overline{\Sigma_{kk}}$ for some $1\leq k\leq d$. Thus $\im(\tau U)$ 
    is singular if and only if $\tau $ assumes one of the (at most
    $2d$) values $\tau = \pm\, \overline{\Sigma_{kk}}$ for some $1\leq k\leq d$. 
    
    This proves both the existence of $\tau\in\T$ as desired as well as the upper bound on the number of exceptions. 

    Note that for $U = A+iB\in\mathrm{U}(d,\C)$, the corresponding orthogonal symplectic matrix $\Ro_U = \begin{psmallmatrix}
    A& B \\ -B & A
\end{psmallmatrix}$ is free if and only if $B = \im(U)$ is invertible. Thus, $\Ro_{\tau U}$ is free if and only if $\im(\tau U)$ is invertible.
\end{proof}

\subsection{Fractional Fourier transforms} 
Before proceeding with the reductions, we take a closer look at the metaplectic operators which project to $\Ro_\Sigma$ for some \emph{diagonal} unitary matrix $\Sigma\in \mathrm{U}(d,\C)$.

In the one-dimensional case, the maximal compact subgroup of $\mathrm{Sp} (1,
\bR )$ consists of the matrices  $\Ro_\sigma = \Ro_{e^{i\alpha}} = \begin{psmallmatrix}
    \cos(\alpha) & \sin(\alpha) \\ 
    -\sin(\alpha) & \cos(\alpha)
\end{psmallmatrix}$ and is thus  isomorphic to the torus $\T$.  
The corresponding metaplectic operator is (up to a phase factor) the
fractional Fourier transform $\F_\sigma= \mu (\Ro _\sigma )$, and in the literature, it is often identified with the corresponding angle $\alpha \in [0,2\pi)$. 
In particular, $\F_1 f = f$, $\F_{-1} f  = f(-\cdot)$, and $\F_i f= \F f$ is the classical Fourier transform \cite{Namias1980, Wiener1929}. 
If  $\sigma\not \in  \{-1,1\}$, then 
$\F _\sigma $ is a quadratic Fourier transform with the explicit
formula 
\begin{equation}\label{eq:frac_integral}
    \F_\sigma f(\omega) = k_\sigma\int_\R f(t) e^{\pi i \cot(\alpha) \omega^2 -2\pi i \csc(\alpha) \omega t +\pi i \cot(\alpha) t^2 } \, dt,
\end{equation}
where 
$k_\sigma \in\C$ is a normalizing constant.

In higher dimensions, the correspondence of diagonal unitary matrices
to fractional Fourier transforms generalizes as follows. 
Given a diagonal matrix $\Sigma= \diag(\Sigma_{jj})_{1\leq j\leq d}\in\mathrm{U}(d,\C)$, let $\Sigma_j$ be the diagonal matrix with $\Sigma_{jj}$ on the $j$-th diagonal entry and otherwise $1$. Then $\mu(\Ro_{\Sigma_j}) = \F_{\Sigma_{j}}$ acts as the partial fractional \ft\  
in the $j$-th coordinate, and 
\begin{equation}
\Sigma = \prod\limits_{j=1}^d \Sigma_j,\qquad
\Ro_\Sigma = \prod\limits_{j=1}^d \Ro_{\Sigma_j}, \qquad
    \mu(\Ro_{\Sigma}) = \F_{\Sigma} = \prod\limits_{j=1}^d \F_{\Sigma_j} .
\end{equation} 
In particular, $\F_{\Sigma}$ acts on elementary tensors as
\begin{equation}\label{eq:_frac_tensor}
    \F_\Sigma (f_1\otimes \dots \otimes f_d) = \F_{\Sigma_{11}}
    f_1\otimes \dots\otimes  \F_{\Sigma_{dd}} f_d.
\end{equation}
All $\F_{\Sigma}$ commute. Furthermore, from the integral
representation \eqref{eq:frac_integral} follows that there exists a constant $c_\Sigma \in\T$ such that 
\begin{equation}\label{eq:frac_inv_rule}
    \F_\Sigma\overline{f} =
    c_{\Sigma}\,\overline{\F_{\overline{\Sigma}} f}\, . 
\end{equation}

We conclude the section with an uncertainty principle for pairs $(f,
\F_\Sigma f)$ for diagonal $\Sigma\in\mathrm{U}(d,\C)$. For $\Sigma =
\diag (i,i, \dots ,i)$, the corresponding operators is just the
ordinary Fourier transform $\F _\Sigma  = \F$, and  the result is just Benedicks's and
Amrein-Berthier's original \up .  
\begin{theorem}\label{thm:Jaming-generalization}
    Let $\Sigma\in\mathrm{U}(d,\C)$ be a diagonal unitary matrix with  a non-zero imaginary part. Then Benedicks's uncertainty principle holds:
    \begin{center}
        If both $f$ and $\F_\Sigma f$ are supported on a set of finite measure, then $f$ is the zero function.
    \end{center}
\end{theorem}
\begin{proof}
Assume $f$ and $\F_\Sigma f$ are supported on a set of finite measure. We want to show that $f$ is the zero function.

Let $\Sigma = A+iB$. If  $B$ is invertible, then 
$\Ro_\Sigma=
\begin{psmallmatrix}
  A & B \\ -B &A
\end{psmallmatrix}
$ is a  free symplectic matrix with a factorization
\begin{equation}
    \Ro_\Sigma = \V_{AB^{-1}} \D_B \J \V_{B^{-1}A} 
    =  \V_{AB^{-1}} \J \D_{B^{-t}}\V_{B^{-1}A} \, ,
  \end{equation}
  according to Proposition~\ref{prop:free_Sp_decomposition}. 
By \eqref{eq:standard_metaplectic}, $\mu(V_Q)$ and $\mu(D_L)$ have no substantial effect on the support, in particular, both $f_1 = \mu(\D_{B^{-t}}\V_{B^{-1}A}) f$ and 
\begin{equation}
    \mu(\V_{AB^{-1}})^{-1} \F_\Sigma f = \F \D_{B^{-t}}\V_{B^{-1}A} f = \F f_1
\end{equation}
are supported on a set of finite measure. By the classical Benedicks's uncertainty principle \cite{Benedicks1985,AmreinBerthier1977}, $f_1$ is the zero function, thus $f = \mu(\D_{B-1}\V_{B^{-t}A})^{-1}f_1\equiv 0$.

If $B$ is not invertible, then  we  consider restrictions $f_x$ to the range of $B$ analogously to the proof of Theorem \ref{thm:new_uncertainty_free}.
\end{proof}
In essence, this proof is the same as Jaming's~\cite{Jaming2022}. Our
contribution is to  make the underlying factorization of metaplectic
operators explicit and visible and then use it for other \up s. 
The numerous \up s for the fractional Fourier transform and the
canonical linear transforms all work with the explicit
formula~\eqref{eq:frac_integral}. Predictably, the proofs tend to become complicated
quickly, and so different methods seem preferable.  

\section{General versions of Benedicks's uncertainty principle}\label{sec:reductions}
In this section,
we provide a criterion for the validity of an uncertainty principle
for $\W_\bA(f,g)$.  
Technically, we will reduce the general case to the special case of free symplectic matrices. 

\begin{lemma}\label{lem:reduce_free}
    Let $\bRo_U\in \Sp $ 
    for some $U\in\mathrm{U}(2d,\C)$. 
    Then the following are equivalent.
    \begin{enumerate}[(i)]
        \item There exist non-zero $f,g\in \ltd$ with 
    $$
    |\supp\, \W_{\bRo_U}(f,g)| <\infty.
    $$
    \item For all $\tau\in\T$
    there exist non-zero $f_\tau,g_\tau\in \ltd$ with 
    $$
    |\supp\, \W_{\bRo_{\tau U}}(f_\tau,g_\tau)| <\infty.
    $$
    \end{enumerate}
\end{lemma}

\begin{proof}
    Let $\tau \in\T$. Then for all $f,g \in \ltd$ holds
    \begin{align}
        |\W_{\bRo_U}(f,g)| 
        & =  |\mu(\bRo_{\tau U}) \mu(\bRo_{\overline{\tau} I}) (f\otimes \overline{g})| \\ 
        & \overset{\eqref{eq:_frac_tensor}}{\underset{\eqref{eq:frac_inv_rule}}{=}}| \mu(\bRo_{{\tau} U})  (\mu(\Ro_{\overline{\tau} I}) f\otimes \overline{\mu(\Ro_{\tau I}) g})|.
    \end{align}
    The equivalence follows by choosing 
    $$
    f_\tau = \mu(\Ro_{\overline{\tau} I}) f,\quad g_\tau = \mu(\Ro_{\tau I}) g.
   \qedhere $$ 
\end{proof}

The characterization of $\W_\bA$, $\bA \in\Sp$, which obey Benedicks's uncertainty principle is a direct consequence of Theorem \ref{thm:new_uncertainty_free}, Lemma \ref{lem:reduce_ortho}, and Lemma \ref{lem:reduce_free}. 

\begin{theorem}\label{thm:main_up_iwasawa}
Let  
    $\bA \in\mathrm{Sp}(4d,\R)$ 
    be a  symplectic matrix with pre-Iwasawa decomposition 
    \begin{equation}
        \bA = \bV_Q \bD_L \bRo_{U}.
    \end{equation}
    Let $\tau \in\T$ such that $\im(\tau U)$ is invertible. 
        If $P = \im(\tau U)^{-1} \re(\tau U)$ is not block-diagonal, then  Benedicks's uncertainty principle holds:
        \begin{center}
If $ \mathrm{W}_\bA (f,g)$ is supported on a set of finite measure for some $f,g\in\ltd$, \\
    then $f\equiv 0$ or $g\equiv 0$.
\end{center}
\end{theorem}

\begin{proof}
Assume that $\mathrm{W}_{\bA} (f_0,g_0)$ with non-zero $f_0,g_0\in L^2(\rd )$ has a
support of finite measure. Then $\mathrm{W}_{\bRo _U} (f_0,g_0)$ has 
also  a support of finite measure by Lemma~\ref{lem:reduce_ortho}. 

Now choose $\tau \in \bT $ such that $\bRo _{\tau U}$ is a free
symplectic matrix; this is possible by
Lemma~\ref{lemma:tau-expansion}. Then by Lemma~\ref{lem:reduce_free} there exist
non-zero $f , g \in L^2(\rd )$, such that $\W_{\bRo_{\tau U}}(f,g)$ is supported on a set of finite
    measure.  
    Set $A=\re
    (\tau U)$ and $B=\im (\tau U)$, and $P=B^{-1}A$. 
    
    By Proposition \ref{prop:free_Sp_decomposition}, 
    \begin{equation}
        \bRo_{\tau U} =A+iB= \bV_{AB^{-1}} \bD_{B} \bJ  \bV_{B^{-1}A}.
    \end{equation}
    By Lemma \ref{lem:reduce_ortho} again, 
     $\W_{\bJ
      \bV_P}(f,g)$ is also  supported on a set of finite measure. Since
    $P=B\inv A$ is not block-diagonal, we can apply  Theorem
    \ref{thm:new_uncertainty_free} and conclude that 
    $f\equiv 0$ or $g\equiv
    0$. Consequently, the same
    conclusion holds for $W_{\bA}(f,g)$ as claimed. 
  \end{proof}

  A special case of Benedicks's \up\ was already obtained
  in~\cite{BayerPhD} for the   class of metaplectic \tf\
  representations of the form $\mathrm{W} (f,g) = \mathcal{F}_2 \bD
  _{L} (f\otimes \bar{g})= \bRo _{\diag (I,iI)} \bD _L (f\otimes \bar{g})$ for $L=
  \begin{psmallmatrix}
    A & B \\ C&D
  \end{psmallmatrix}$ with invertible blocks $B$ and $D$. These \tfr s
  can be related directly to the \stft\ and require less work.

By Lemma \ref{lemma:tau-expansion}, there always exists a $\tau \in
\T$ such that $\im(\tau U)$ is invertible, so Theorem
\ref{thm:main_up_iwasawa} can be applied to every symplectic matrix. 
However, for  an arbitrary  symplectic matrix $\bA \in\Sp$, it is not
obvious how to check the assumptions of Theorem~\ref{thm:main_up_iwasawa} directly for
$\bA $. In the following we elaborate on this question and will
present an algorithm to decide whether or not a Benedicks-type \up\
holds for $W_{\bA }$. 

\begin{lemma}\label{lem:diagonal_P_tau_free}     Let $U = A+iB\in\mathrm{U}(2d,\C)$ be a unitary matrix with an invertible imaginary part $B\in\mathrm{GL}(2d,\R)$ and let $P = B^{-1}A$. Then the following are equivalent.
    \begin{enumerate}[(i)]
    \item 
    There exist symmetric matrices $P_{1}, P_{2} \in\R^{d\times d}$
    such that $P$ is block-diagonal, 
    $$
    P =\diag(P_{1}, P_{2}).
    $$
        \item There exist orthogonal matrices $W\in\mathrm{O}(2d,\R)$,
          $V_1, V_2\in\mathrm{O}(d,\R)$,  and diagonal unitary matrices $
     \Sigma_1,\Sigma_2 \in\mathrm{U}(d,\C)$ with 
     \begin{equation}\label{eq:lem_block_char}
        U = W\cdot \diag(\Sigma_1,\Sigma_2) \cdot \diag(V_1,V_2)^t = W\cdot \diag(\Sigma_1 V_1^t, \Sigma_2 V_2^t).
    \end{equation}
    \item The unitary matrix $U^t U$ is block-diagonal.
    \end{enumerate}
\end{lemma}
\begin{proof}
    We will show the chain of implications (ii) $\Rightarrow$ (iii) $\Rightarrow$ (i) $\Rightarrow$ (ii).

    A simple computation shows that (ii) implies (iii).

 (iii) $ \Rightarrow  $ (ii)     Since $U = A+iB$ is unitary, its real and  imaginary parts satisfy
    \begin{equation}
        I = \re(U^* U) = A^t A +B^tB\qquad \text{and}\qquad 0 = \im(U^*U)=A^tB-B^tA.
    \end{equation}
    Then  $B^tA$ is symmetric  and 
    \begin{equation}
        \begin{split}
            U^tU = A^tA-B^tB+ i(A^tB+B^tA) = I-2B^tB+2iB^tA.
        \end{split}
    \end{equation}
    By assumption, $U^tU$ is block-diagonal, thus both  its real part $I-2B^tB$ and its imaginary part $2B^tA$ are block-diagonal. Therefore, $B^tB$ is block-diagonal, and invertible by assumption on $B$.
    From this, we conclude that
    \begin{equation}
        P = B^{-1}A = B^{-1}B^{-t}B^t A = (B^tB)^{-1}\,(B^tA)
    \end{equation}
    is block-diagonal and symmetric  as a product of block-diagonal,
    symmetric  matrices.

    (i) $\Rightarrow$ (ii). Since the matrix $P=B^{-1}A$ is symmetric, the matrix $I+P^2$ is positive definite, and the following matrix is well-defined, with invertible imaginary part:
    \begin{equation}
        U' = (I+P^2)^{-1/2}(P+iI) = A'+iB'.
    \end{equation}
    The matrix $U'$ is also unitary, because  $U'U'^*= U'^* U' = I$.
    Furthermore, 
    we also have  $B'^{-1}A' = (1+P^2)^{1/2} [ (1+P^2)^{-1/2} P] =  P$.
    Since $B,B'$ are invertible, $W'= B B'^{-1}$ is invertible, 
    and consequently,
    \begin{equation}
        A = BP = W'B'P = W'A'.
    \end{equation}
    Thus $U = W' U'$. By assumption, $P$ is block-diagonal, hence $U'$
    is block-diagonal as well.  
    By Lemma \ref{lem:unitary_decomp}, there exist matrices $W_1,W_2,V_1,V_2\in\mathrm{O}(d,\R)$ and diagonal matrices $\Sigma_1,\Sigma_2 \in\mathrm{U}(d,\C)$ such that 
    \begin{equation}
        U ' = \begin{pmatrix}
            W_1 & 0 \\ 0 & W_2 
        \end{pmatrix}\cdot \begin{pmatrix}
            \Sigma_1 & 0 \\ 0 & \Sigma_2
        \end{pmatrix} 
        \cdot
        \begin{pmatrix}
            V_1 & 0 \\ 0 & V_2
        \end{pmatrix}
    \end{equation}
    We set 
    \begin{equation}
        W = W'\cdot \diag(W_1,W_2)
    \end{equation}
    to obtain the claimed factorization.
\end{proof}

We now formulate the main result.

\begin{theorem}\label{thm:new_uncertainty}     Let  
    $\bA \in\mathrm{Sp}(4d,\R)$ 
    be a symplectic matrix with pre-Iwasawa decomposition 
    \begin{equation}
        \bA = \bV_Q \bD_L \bRo_{U}.
    \end{equation}
    Then the following alternative holds.
    \begin{enumerate}[(i)]
        \item 
        If $U^t U$ is not block-diagonal, then Benedicks's uncertainty principle holds:
        \begin{center}
    If $ \mathrm{W}_\bA (f,g)$ is supported on a set of finite measure for some $f,g\in\ltd$, \\
    then $f\equiv 0$ or   $g\equiv 0$.
\end{center}
        \item If $U^t U$ is block-diagonal, then there exist non-zero $f,g\in\ltd$, 
          such that $\W_\bA(f,{g})$ is supported on a compact set, in particular, on a set of finite measure.
    \end{enumerate}
\end{theorem}

\begin{proof} 
  (i)   The \up \ is a direct consequence of Theorem
    \ref{thm:main_up_iwasawa} and Lemma
    \ref{lem:diagonal_P_tau_free}. To see this, let $f,g\in \ltd$ such
    that $\W_\bA(f,g)$ is supported on a set of finite measure. 

    By Lemma \ref{lemma:tau-expansion}, there exists a $\tau \in \T$ such that $\im(\tau U)$ is invertible. By assumption, the matrix
    \begin{equation}
        (\tau U)^t (\tau U) = \tau^2 U^tU
    \end{equation}
    is not-block-diagonal. Consequently, 
    by Lemma \ref{lem:diagonal_P_tau_free}, the symmetric matrix $P =
    \im(\tau U)^{-1}\re(\tau U)$ is not block-diagonal either. Theorem
    \ref{thm:main_up_iwasawa} now implies that $f\equiv 0$ or $g\equiv
    0$. 
    
(ii)     We now prove the second statement. If $U^t U$ is block-diagonal, so is the matrix $(\tau U)^t (\tau U) = \tau^2 U^tU$ for any $\tau \in\T$.
Using  Lemma \ref{lemma:tau-expansion}, we can choose  $\tau\in\T$
    such that $\im(\tau U)$ is invertible and apply  Lemma
    \ref{lem:diagonal_P_tau_free} to write $\tau U$ as   
     \begin{equation}
        \tau U = 
        W\cdot \diag(\Sigma_{1,\tau} V_1^t,
        \Sigma_{2,\tau} V_2^t) 
      \end{equation}
    for some  $W\in\mathrm{O}(2d,\R)$, $V_1, V_2\in\mathrm{O}(d,\R)$ and diagonal unitary matrices $
     \Sigma_{1,\tau},\Sigma_{2,\tau} \in\mathrm{U}(d,\C)$. 
    Set $\Sigma_1 = \overline{\tau}\Sigma_{1,\tau}$ and $\Sigma_2 = \overline{\tau}\Sigma_{2,\tau}$, then
    \begin{equation}
        U = 
        W\cdot \diag(\Sigma_1 V_1^t, \Sigma_2 V_2^t).
    \end{equation}
    By Remark \ref{rem:embedding}, 
    for all $f,g\in\ltd$, 
    \begin{align}
        |\W_\bA(f,{g})| 
        & = |\mu(\bV_Q)\mu(\bD_L)\mu(\bD_W)\mu(\bRo_{\diag(\Sigma_1 V_1^t, \Sigma_2 V_2^t)}) \, (f\otimes \overline{g})|\\
        & = |\det L|^{-1/2}\,|(\mu(\Ro_{\Sigma_1 V_1^t})f\otimes \overline{\mu(\Ro_{\overline{\Sigma_2} V_2^t})g})(W^t L^{-1}\, \cdot \,)|.
    \end{align}
 Choose non-zero  $f_0,g_0\in\ltd$ with compact support, then 
 the functions
    $$
    f = \mu(\Ro_{\Sigma_1 V_1^t})^{-1}f_0 
   \quad \text{ and } \quad   g = \mu(\Ro_{\overline{\Sigma_2} V_2^t})^{-1}g_0 
    $$
    are non-zero and square-integrable,    since $\mu(\Ro_{\Sigma_1 V_1^t}),
   \mu(\Ro_{\overline{\Sigma_2} V_2^t})$ are unitary operators. Their
   $\bA$-Wigner distribution is 
    $$
    |\W_\bA( f,g)| = |\det L|^{-1/2}\,|(f_0\otimes \overline{g_0})(W^t
    L^{-1}\,\cdot\,)| 
    $$
    and has compact support as well. 
\end{proof}

Theorem~\ref{thm:new_uncertainty} characterizes all symplectic
matrices for which the metaplectic \tf\ distribution $W_{\bA }$
satisfies a Benedicks-type \up\  
and yields a simple  algorithm to implement this characterization. 

\vspace{3mm}

\begin{algorithm}[H]
    Input: $\bA \in \Sp$.
    \begin{enumerate}
        \item Compute the pre-Iwasawa decomposition $\bA = \bV_Q \bD_L \bRo_{U}.$
        \item Compute $U^t U$. 
        \item If $U^t U$ is block-diagonal, then there exist non-zero
          $f,g \in\ltd$, \\ 
        such that $\W_\bA(f,{g})$ is supported on a compact set. 
        \item Otherwise, whenever $\W_\bA(f,{g})$ is supported on a set of finite measure, then $f\equiv 0$ or $g\equiv 0$. 
    \end{enumerate}
\end{algorithm}
To determine  the  pre-Iwasawa decomposition one can use the explicit formulas
given in\cite[Sec.~2.2.2]{Gosson2006}. It would be interesting to understand the geometric meaning of the
condition on $U^tU$.

\section{Quadratic uncertainty principles}\label{sec:auto} 
This section is dedicated to uncertainty principles with an additional restriction on the pair $(f,g)$, namely, $g= f$. The 
motivation stems from the Rihaczek distribution 
$$
\mathrm{R}(f, g)(x,\omega) = e^{-2\pi i x \cdot  \omega}f(x) \overline{\F g(\omega)},
$$
in which case we can clearly find non-zero functions $f,g\in\ltd$, 
such that $\mathrm{R}(f, g)$ is compactly supported. However, with the additional assumption $g = f$, $\mathrm{R}(f,f)$ cannot be supported on a set of finite measure unless $f$ is the zero function. This is precisely a reformulation of Benedicks's \up~\cite{Benedicks1985, AmreinBerthier1977}.

Before we state the main result, we prove a version of   Lemma
\ref{lem:diagonal_P_tau_free} which is better  suited for quadratic
\tf\ distributions.

\begin{lemma}\label{lem:diagonal_P_tau_auto} 
    Let $U = A+iB\in\mathrm{U}(2d,\C)$ be a unitary matrix. Then the following statements are equivalent.
    \begin{enumerate}[(i)]
        \item There exist an orthogonal matrix $W\in\mathrm{O}(2d,\R)$ and a unitary matrix $\Delta\in\mathrm{U}(d,\C)$ such that
     \begin{equation}\label{eq:lem_block_char_auto}
        U = W\cdot \diag(\Delta,\overline{\Delta}). 
    \end{equation}
    \item There exist an orthogonal matrix $W\in\mathrm{O}(2d,\R)$ and
      unitary matrices $\Delta_1, \Delta_2\in\mathrm{U}(d,\C)$ such
      that $\Delta_1 \Delta_2^t \in \mathrm{U}(d,\C)$ is real-valued and
     \begin{equation} 
        U = W\cdot \diag(\Delta_1,\Delta_2). 
    \end{equation} 
    \item There exists a unitary matrix $\boldsymbol{\Delta}\in\mathrm{U}(d,\C)$ such that
    \begin{equation} \label{eq:lem_block_char_auto_2}
        U^t U = \diag(\boldsymbol{\Delta},\overline{\boldsymbol{\Delta}}).
    \end{equation}
    \end{enumerate}
\end{lemma}
\begin{proof}
    We show the following chain of implications
    \begin{equation}
        \text{(i)}\ \Rightarrow \ 
        \text{(iii)} \ \Rightarrow\ 
        \text{(ii)}\ \Rightarrow \ \text{(i)}.
    \end{equation}
    A simple computation shows that (i) implies (iii) with
    $\boldsymbol{\Delta } = \Delta ^t\Delta $. 

(iii) $\Rightarrow$ (ii)    
Let $\tau\in\T$
    such that $\im(\tau U)$ is invertible according to  Lemma \ref{lemma:tau-expansion}. By Lemma \ref{lem:diagonal_P_tau_free},  there exist $W\in\mathrm{O}(2d,\R)$, $V_1, V_2\in\mathrm{O}(d,\R)$ and diagonal unitary matrices $
     \Sigma_{1,\tau},\Sigma_{2,\tau} \in\mathrm{U}(d,\C)$ with  
     \begin{equation}
        \tau U = W\cdot \diag(\Sigma_{1,\tau}V_1,\Sigma_{2,\tau}V_2^t) .
    \end{equation}
    Set  
    $\Delta_1 = \bar{\tau}\Sigma_1 V_1^t,\, \Delta_2 = \bar{\tau } \Sigma_2V_2^t\in \mathrm{U}(d,\C)$. Then
    \begin{equation}
        U = 
        W \cdot \diag(\Delta_1,\Delta_2).
    \end{equation}
    By assumption,
    \begin{equation}
    \diag(\boldsymbol{\Delta}, \overline{\boldsymbol{\Delta}}) = U^tU = \begin{pmatrix}
        \Delta_1^t\Delta_1 & 0 \\ 0 & \Delta_2^t \Delta_2
    \end{pmatrix},
    \end{equation}
    hence 
    \begin{equation}\label{eq:lem_bold_delta_ii}
         \overline{\Delta_1^t \Delta_1} = \Delta_2^t \Delta_2  \, .
    \end{equation}
    We need to show that $\Delta_1 \Delta_2^t$ is real-valued. 
    Indeed, by multiplying \eqref{eq:lem_bold_delta_ii} with
    $\Delta_1$ from the left and with $\Delta_2^{-1} = \Delta_2^* =
    \overline{\Delta_2^t}$ from the right, we obtain 
    \begin{equation}
        \Delta_1 \Delta_1^* \overline{\Delta_1} {\Delta_2^*} =
        \Delta_1 \Delta_2^t \Delta_2 \Delta_2^* \, ,
    \end{equation}
or equivalently, $\overline{\Delta_1 \Delta_2^t} = \Delta_1\Delta_2^t$,    as claimed.

(ii) $\Rightarrow$ (i) The assumption says that     there exist $W'\in\mathrm{O}(2d,\R)$ and
    $\Delta_1,\Delta_2\in\mathrm{U}(d,\C)$  such that
    \begin{equation}
        U = W'\cdot \diag(\Delta_1,\Delta_2)
    \end{equation}
and the matrix  $W_0 = \Delta_1\Delta_2^t\in \R^{d\times d}$  is
real-valued and hence orthogonal. Using $\Delta_1^{-1} =
\overline{\Delta_1^t}= \Delta _1^*$,
    \begin{equation}
        \Delta_2 = (\Delta_1^{-1}W_0)^t = W_0^t\overline{\Delta_1}.
    \end{equation}
    We set $W = W'\cdot \diag(I,W_0^t) \in \mathrm{O}(2d,\R)$ and $\Delta = \Delta_1\in\mathrm{U}(d,\C)$. Then
    \begin{equation}
        U = W\cdot \diag(\Delta,\overline{\Delta})
    \end{equation} 
    and (i) follows.
\end{proof}

Before we formulate the main result for quadratic metaplectic Wigner
distributions, we prove an extension of
Theorem~\ref{thm:Jaming-generalization} for  a pair
$(\mu(\Ro_{U_1})f,\mu(\Ro_{U_2})f)$. 
\begin{theorem}\label{thm:new-uncertainty-pairs}
Let $\Ro_{U_1},\Ro_{U_2}\in \Spp$ be a orthogonal symplectic matrices such that the matrix $U_1 U_2^*$ is not real-valued. Then Benedicks's uncertainty principle holds for the pair $(\mu(\Ro_{U_1})f,\mu(\Ro_{U_2})f)$:
\begin{center}
If both $\mu(\Ro_{U_1})f$ and $\mu(\Ro_{U_2})f$ are supported on a set of finite measure, \\
    then $f\equiv 0$. 
\end{center}
\end{theorem}
\begin{proof}
    The proof idea relies on Theorem \ref{thm:Jaming-generalization}
    and the matrix decomposition of Lemma~\ref{lem:unitary_decomp}. 

    Assume $f\in \ltd$ such that both $\mu(\Ro_{U_1})f$ and $\mu(\Ro_{U_2})f$ are supported on a set of finite measure. We aim to show that $f$ is the zero function.
    
    We set $U = U_1U_2^*$ and $f_1 =\mu(\Ro_{U_2})f$. 
    Then  by Remark \ref{rem:embedding}
    \begin{equation}
        |\mu(\Ro_{U_1})f| =| \mu(\Ro_{U_1})\mu(\Ro_{U_2})^{-1} f_1 |
        = |\mu(\Ro_{U_1}\Ro_{U_2}^{-1}) f_1|
        =|\mu(\Ro_{U_1 U_2^*}) f_1|
        =|\mu(\Ro_{U}) f_1 |\, .
    \end{equation}
    By construction, $f_1$ and $\mu(\Ro_{U}) f_1$ are supported on a set of finite measure. 

We now decompose $U$ with  Lemma \ref{lem:unitary_decomp} as  $U =
W\Sigma V^t$ with 
    $W,V\in\mathrm{O}(d,\R)$ and a diagonal matrix
    $\Sigma\in\mathrm{U}(d,\C)$. By
    Lemma~\ref{rem:Sp_rechenregel}, $\Ro_W =\D_W$ and $\Ro_{V^t}     =\D_{V^t}$,  
    and $\mu(\D_W),\, \mu(\D_{V^t})$ act as  dilation operators,
    therefore     both $f_2 = \mu(\D_{V^t})f_1$ and
    \begin{equation}
    \mu(\Ro_{\Sigma}\Ro_{V^t}) f_1 =  \mu(\Ro_{\Sigma}) f_2 =\F_\Sigma f_2
    \end{equation}
    are supported on a set of finite measure. Furthermore, since $W,V$
    are real-valued but $U= U_1U_2^*$ is not real-valued by
    assumption, we conclude that $\mathrm{Im}\, (\Sigma ) \neq 0$.
    By Theorem \ref{thm:Jaming-generalization} (Benedicks's
    uncertainty principle for fractional Fourier transforms), $f_2$ is the zero function, hence
    \begin{equation}
        f = \mu(\Ro_{U_2})^{-1} f_1 = \mu(\Ro_{U_2})^{-1} \mu(\D_{V^t})^{-1} f_2 \equiv 0,
    \end{equation}
    which concludes the proof. 
\end{proof}
Finally, we prove the \up\  for quadratic metaplectic Wigner distributions.
\begin{theorem}\label{thm:auto_uncertainty} 
    Let $\bA \in \Sp$  be a symplectic matrix with pre-Iwasawa decomposition 
    \begin{equation}
        \bA = \bV_Q \bD_L \bRo_{U}.
    \end{equation}
    Then the following alternative holds. 
    \begin{enumerate}[(i)]
    \item If $U^tU \neq \diag(\Delta, \overline{\Delta})$ for all unitary matrices $\Delta\in\mathrm{U}(d,\C)$, then Benedicks's uncertainty principle holds:
    \begin{center}
    If $\mathrm{W}_\bA(f,f)$ is supported on a set of finite measure for some $f\in\ltd$, \\
    then $f\equiv 0$.
\end{center}
        \item 
        If $U^t U  =\diag(\Delta, \overline{\Delta})$ for some unitary matrix $\Delta\in\mathrm{U}(d,\C)$, 
        then there exists a non-zero function $f\in\ltd$, such that the support of $\W_\bA(f,f)$ is compact, in particular, has finite measure.
    \end{enumerate}
\end{theorem}

\begin{proof}
  By Lemma \ref{lem:reduce_ortho}, we may assume that $Q =0$ and $L=
  I$ and  that $\bA = \bRo_{U}$.
  
  We first prove (i) and assume that $\W_\bA(f,f)=
  \W_{\bRo_{U}}(f,f) $ is supported on a set of finite measure for $f\in L^2(\rd )$. We  distinguish two cases.

  \textbf{Case 1.}  $U^t U$ is not block-diagonal. Then  Theorem
  \ref{thm:new_uncertainty}  with $g= f$ implies that $f=0$.

    \textbf{Case 2.} $U^t U$ is block-diagonal, but $U^tU \neq \diag
    (\Delta, \Delta ^t)$. Then by Lemma~\ref{lem:diagonal_P_tau_auto}, there exists a matrix $W\in\mathrm{O}(2d,\R)$ and unitary matrices $\Delta_1, \Delta_2\in\mathrm{U}(d,\C)$ with $\Delta_1\Delta_2^t\notin \R^{d\times d}$ such that
    $$
    U = W\cdot \diag(\Delta_1,\Delta_2).
    $$
   Recall that by Lemma~\ref{rem:Sp_rechenregel} $\mu(\bRo_W) =
   \mu(\bD_W)$ and by Remark \ref{rem:embedding} the action of
   $\mu(\bRo_{\diag(\Delta_1, \Delta_2)})$ tensorizes. Consequently,
   \begin{align*}
     |\mu (\bRo_{U})(f\otimes \bar{f}) |&= |\mu (\bD _W) \mu (\bRo_{\diag
     (\Delta _1, \Delta _2)} (f\otimes \bar{f})| \\
&=\big|\mu (\bD _W) \big( \mu (\Ro_{\Delta _1} f \otimes \Ro_{\Delta_2}\bar{f})\big)\big|\, .
   \end{align*}
Since $\mu (\bD _W)F$ amounts to a rotation of the support of $F$, we
may ignore it and are left with the fact that the tensor product $\mu
(\Ro_{\Delta _1} )f \otimes \mu (\Ro_{\Delta   _2} )\bar{f}= \mu
(\Ro_{\Delta _1} f \otimes \overline{\Ro_{\bar{\Delta} 
  _2}f)}$ is supported on a set of finite measure.  
 
    By assumption, $U^tU$ is not block-diagonal  of the form
 $\diag (\Delta, \bar{\Delta })$, thus condition (ii) of
 Lemma~\ref{lem:diagonal_P_tau_auto} says that $\Delta _1 \Delta _2^t=
 \Delta _1 \bar{\Delta }_2^*$
 cannot be real-valued. This is precisely the assumption of
 Theorem~\ref{thm:new-uncertainty-pairs}, and we conclude that $f=0$.

    We now prove (ii). By Lemma \ref{lem:diagonal_P_tau_auto}, there
    exists an orthogonal matrix $W\in\mathrm{O}(2d,\R)$ and a unitary
    matrix $\Delta \in\mathrm{U}(d,\C)$ such that 
    $$
    U = W\cdot \diag(\Delta,\overline{\Delta}).
    $$
Choose non-zero $f _0\in\ltd$ with compact support
    and  set 
    $
    f = \mu(\Ro_{{{\Delta}}})^{-1}f_0.
    $
    Then $f\in\ltd$ and $f\not\equiv 0$ because 
    $\mu(\Ro_{{{\Delta}}})^{-1}$ is a unitary operator.
 Using again that $\bRo_W = \bD_W$, we obtain 
    \begin{align}
        |\W_\bA(f,f)|  
      & = \mu (\bD_W) |\mu (\bRo_{\diag (\Delta , \bar{\Delta })}) (f
        \otimes \bar{f})| \\
&= \mu (\bD_W) ( |\mu(\Ro_{{\Delta}}) f| \otimes
      |\overline{\mu(\Ro_{{\Delta}}) f}|) \\ 
   &        = \mu (\bD _W) (|f_0| \otimes |{f_0}|),
    \end{align}  
and this function  is clearly supported on a compact set.
\end{proof}

\begin{remark}
   Let $\iota f $ be either complex conjugation $f \mapsto \bar{f}$
   or the involution $f \mapsto f^*, f^*(t) = \bar{f}(-t)$. Then
   Theorem~\ref{thm:auto_uncertainty} also holds for the quadratic
   \tf\ distribution $W_{\A } (f, \iota f)$. The condition   
   $\Delta_1\Delta_2^t\in\R^{d\times d}$ needs to be replaced by the
   condition  $\Delta_1\Delta_2^{-1} =\Delta_1 \overline{\Delta_2^t}
   \in\R^{d\times d}$, and the proof is the same. 
\end{remark}

\section{Examples}\label{sec:examples} 
In this section, we show how to apply Theorems~\ref{thm:new_uncertainty} and~\ref{thm:auto_uncertainty}, and
we  recover the known uncertainty principles for the \stft, the
ambiguity function, and the $\tau$-Wigner distribution. The symplectic
matrices associated to the  \stft\ and the $\tau $-Wigner distributions
were already derived in ~\cite{CorderoRodino2023}, here we test them with
regard to the validity of the \up . 

The short-time Fourier transform  is given by
    $$
    \mathrm{V}_{g}f (x,\omega) = \int_\Rd f(t)
    \overline{g(t-x)}e^{2\pi i \omega \cdot  t} \, dt.
    $$
Let     
    \begin{equation}
        L_{\scaleto{\mathrm{STFT}}{4.5pt}} = 
        \begin{pmatrix} I & -I \\ I & 0 \end{pmatrix} \in\mathrm{GL}(2d,\R)
        \quad \text{with inverse}\quad 
        L_{\scaleto{\mathrm{STFT}}{4.5pt}}^{-1} = 
        \begin{pmatrix}  
        0 & I \\ -I & I \end{pmatrix}
    \end{equation}
    (the matrix $L_{\scaleto{\mathrm{STFT}}{4.5pt}}^{-1}$ corresponds
    to the linear map $(x,t)\mapsto (t,t-x)$). Since the coordinate
    change by $L_{\scaleto{\mathrm{STFT}}{4.5pt}}$ and the partial
    Fourier transform are metaplectic operators, 
     $\Vb_g f$ can be written as 
    the $\bA$-Wigner distribution
    $$
    \Vb_g f = \mu(\bA _{\scaleto{\mathrm{STFT}}{4.5pt}})(f\otimes \overline{g}) = \mu(\bRo_{\diag(I, iI)}\bD_{L_{\mathrm{STFT}}})(f\otimes \overline{g})\text{\ \ with the matrix}
    $$
    $$
    \bA _{\scaleto{\mathrm{STFT}}{4.5pt}} = \bRo_{\diag(I,
      iI)}\bD_{L_{\mathrm{STFT}}} = \begin{pmatrix}
        I & -I & 0 & 0 \\
        0 & 0 & I & I \\
        0 & 0 & 0 & -I \\
        -I & 0 & 0 & 0
    \end{pmatrix}.
    $$
    The pre-Iwasawa decomposition of $\bA _{\scaleto{\mathrm{STFT}}{4.5pt}}$ is given by
    \begin{equation}\label{eq:STFT-iwasawa}
        \bA _{\scaleto{\mathrm{STFT}}{4.5pt}} = \bV_{-\frac{1}{2}\begin{psmallmatrix}
            0 & I \\ I & 0
        \end{psmallmatrix}}
        \cdot \bD_{\sqrt{2} I}
        \cdot \bRo_{U_{\mathrm{STFT}}},
        \quad U_{\scaleto{\mathrm{STFT}}{4.5pt}} = \frac{1}{\sqrt{2}}\begin{pmatrix}
            I & -I \\ iI & iI
        \end{pmatrix}  \in\mathrm{U}(2d,\C).
    \end{equation}
    The deciding product is
    $$
    U_{\scaleto{\mathrm{STFT}}{4.5pt}}^t U_{\scaleto{\mathrm{STFT}}{4.5pt}}= \begin{pmatrix}
         0 & - I\\ 
        -I & 0
    \end{pmatrix},
    $$
    which is not block-diagonal, thus the Benedicks-type uncertainty
    principle holds for the \stft , which is known from   \cite{Jaming1998,Janssen1998_supp, GroechenigZimmermann2001,Wilczok2000}.

    The ambiguity function  is a more symmetric version of the \stft\
    and is defined as
    \begin{align*}
         \AF(f,g)(x,\omega) &=  \int_{\Rd}
        f(t+\tfrac{x}{2})\overline{g(t-\tfrac{x}{2})}e^{-2\pi i
                              \omega\cdot t}\, dt \\
     & =  e^{\pi i x\cdot \omega} \Vb_g f(x,\omega) 
        = \mu(\bV_{\frac{1}{2}\begin{psmallmatrix}
            0 & I \\ I & 0
        \end{psmallmatrix}})
        \Vb_g f(x,\omega),\quad x,\omega\in\Rd,
    \end{align*}
    thus is given by $\AF(f,g) = \mu(\bA_{\scaleto{\mathrm{A}}{4.5pt}})(f,g)$, where 
    \begin{equation}
         \bA_{\scaleto{\mathrm{A}}{4.5pt}} 
         = \bV_{\frac{1}{2}\begin{psmallmatrix}
            0 & I \\ I & 0
        \end{psmallmatrix}} 
        \cdot  \bA _{\scaleto{\mathrm{STFT}}{4.5pt}} 
       = \begin{pmatrix}
            I & -I & 0 & 0 \\
            0 & 0 & I & I \\
            0 & 0 & \tfrac{1}{2}I & -\tfrac{1}{2}I  \\
            -\tfrac{1}{2}I  & -\tfrac{1}{2}I & 0 & 0
        \end{pmatrix}.
    \end{equation}
    By \eqref{eq:STFT-iwasawa}, the pre-Iwasawa decomposition of $\bA_{\scaleto{\mathrm{A}}{4.5pt}}$ is given by
    \begin{equation}
       \bA_{\scaleto{\mathrm{A}}{4.5pt}} = 
       \bV_0 \cdot\bD_{\sqrt{2}I}
        \cdot \bRo_{U_{\mathrm{A}}},
        \quad U_{\scaleto{\mathrm{A}}{4.5pt}} = \frac{1}{\sqrt{2}}\begin{pmatrix}
            I & -I \\ iI & iI
        \end{pmatrix} \in\mathrm{U}(2d,\C),
    \end{equation}
    and as above 
     $$
    U^t_{\scaleto{\mathrm{A}}{4.5pt}} U_{\scaleto{\mathrm{A}}{4.5pt}}= \begin{pmatrix}
         0 & - I\\ 
        -I & 0
    \end{pmatrix},
    $$
    hence it obeys the Benedicks-type uncertainty principle
    \cite{Jaming1998,Janssen1998_supp,
      GroechenigZimmermann2001,Wilczok2000}.
    
Finally, the $\tau$-Wigner distribution is given by
    \begin{equation}
    \begin{split}
    \W_\tau (f,g)(x,\omega) &= \int_{\Rd} f(x+\tau t)\overline{g(x-(1-\tau)t)}e^{-2\pi i \omega t}\, dt.
    \end{split}
    \end{equation}
    This family of time-frequency representations includes the Wigner distribution for $\tau=\frac{1}{2}$ and 
    the Rihaczek distribution for $\tau = 0$.
    By choosing
    \begin{equation}
    L_\tau = 
    \begin{pmatrix} I-\tau I & \tau I \\ I &-I  \end{pmatrix}  \in\mathrm{GL}(2d,\R)
    \quad \text{with inverse}\quad
    L_\tau^{-1} = 
    \begin{pmatrix}  I & \tau I \\  I & -(1-\tau)I \end{pmatrix},
\end{equation}
    $\W_{\tau}(f,g)$ can be written as
    $$
    \W_{\tau}(f,g) =  \mu(\bRo_{\diag(I, iI)}\bD_{L_{\tau}})(f\otimes
    \overline{g}) = \mu(\bA _{\tau})(f\otimes \overline{g}) \text{\ \ with the matrix}
    $$
    $$
    \bA_{\tau} = \begin{pmatrix}
        (1-\tau) I & \tau I & 0 & 0 \\
        0 & 0 & \tau & -(1-\tau)I \\
        0 & 0 & I & I \\
        -I & I & 0 & 0
    \end{pmatrix}.
    $$
    With the parameter $ \alpha_\tau \coloneqq (1-2\tau+2\tau^2)^{1/2} = ((1-\tau)^2+\tau^2)^{1/2}$
    the pre-Iwasawa decomposition of $\bA _{\tau}$ is given by
    \begin{equation}
        \bA _\tau = \bV_{-\frac{(1-2\tau)}{\alpha_\tau^2}\begin{psmallmatrix}
            0 & I \\ I & 0
        \end{psmallmatrix}} 
        \cdot \bD_{\alpha_\tau I} 
        \cdot \bRo_{U_{\tau}},\quad  U_{\tau} = \frac{1}{\alpha_\tau} \begin{pmatrix}
            1-\tau & \tau \\  i\tau & -i(1-\tau)
        \end{pmatrix}  \in\mathrm{U}(2d,\C).
    \end{equation}
    The deciding product is given by 
    $$
    U_{\tau}^t U_{\tau} = \frac{1}{\alpha_\tau^2} \begin{pmatrix}
       (1-2\tau) I &  2(1-\tau)\tau I \\
       2(1-\tau)\tau I & -(1-2\tau) I
    \end{pmatrix}.
    $$
    The matrix is block-diagonal if and only if  $\tau = 0 $ or $\tau
    = 1$, thus a Benedicks-type \up\ holds for the $\tau $-Wigner
    distribution  precisely when $\tau \neq 0,1$, and fails for the
    Rihaczek distributions $\tau = 0 $ and $\tau =1$.
    
    Concerning the uncertainty principle for quadratic time-frequency representations in Theorem \ref{thm:auto_uncertainty}, if $\tau\in\{0,1\}$, then 
    \begin{equation}
        U_{0}^t U_{0} =  \begin{pmatrix}
       I & 0 \\ 0 & - I
    \end{pmatrix},
    \qquad
    U_{1}^t U_{1} =  \begin{pmatrix}
       -I &  0 \\ 0 & I
    \end{pmatrix}.
    \end{equation}
These are not of the form     
$U_{\tau}^t U_{\tau} = \diag(\boldsymbol{\Delta},\overline{\boldsymbol{\Delta}})$. Hence, for all $\tau\in\T$,
    the only function $f\in\ltd$ such that $\W_\tau(f,{f})$ is supported on a set of finite measure is the zero function.


\begin{thebibliography}{10}

\bibitem{AmreinBerthier1977}
{\sc Amrein, W.~O., and Berthier, A.~M.}
\newblock On support properties of {$L\sp{p}$}-functions and their {F}ourier
  transforms.
\newblock {\em J. Funct. Anal.  24}, 3 (1977), 258--267.

\bibitem{BayerPhD}
{\sc Bayer, D.}
\newblock {\em {Bilinear Time-Frequency Distributions and Pseudodifferential
  Operators}}.
\newblock PhD thesis, University of Vienna, Vienna, Austria, 2010.

\bibitem{Benedicks1985}
{\sc Benedicks, M.}
\newblock On {F}ourier transforms of functions supported on sets of finite
  {L}ebesgue measure.
\newblock {\em J. Math. Anal. Appl. 106}, 1 (1985), 180--183.

\bibitem{BenyiOkodjou2020}
{\sc B\'{e}nyi, A., and Okoudjou, K.~A.}
\newblock {\em Modulation spaces---with applications to pseudodifferential
  operators and nonlinear {S}chr\"{o}dinger equations}.
\newblock Applied and Numerical Harmonic Analysis. Birkh\"{a}user/Springer, New
  York, 2020.


\bibitem{BonamiDemange2006}
{\sc Bonami, A., and Demange, B.}
\newblock A survey on uncertainty principles related to quadratic forms.
\newblock {\em Collect. Math.}, Vol. Extra (2006), 1--36.

\bibitem{CarmonaEtAl1998}
{\sc Carmona, R., Hwang, W.-L., and Torr\'{e}sani, B.}
\newblock {\em Practical time-frequency analysis}, vol.~9 of {\em Wavelet
  Analysis and its Applications}.
\newblock Academic Press, Inc., San Diego, CA, 1998.
\newblock Gabor and wavelet transforms with an implementation in S, With a
  preface by Ingrid Daubechies.

\bibitem{Cordero_Rodino}
{\sc Cordero, E., and Rodino, L.}
\newblock {\em Time-frequency analysis of operators}, vol.~75 of {\em De
  Gruyter Studies in Mathematics}.
\newblock De Gruyter, Berlin, 2020.

\bibitem{CorderoRodino2022}
{\sc Cordero, E., and Rodino, L.}
\newblock Wigner analysis of operators. {P}art {I}: {P}seudodifferential
  operators and wave fronts.
\newblock {\em Appl. Comput. Harmon. Anal. 58\/} (2022), 85--123.

\bibitem{CorderoRodino2023}
{\sc Cordero, E., and Rodino, L.}
\newblock Characterization of modulation spaces by symplectic representations
  and applications to {S}chr\"{o}dinger equations.
\newblock {\em J. Funct. Anal. 284}, 9 (2023), Paper No. 109892, 40.

\bibitem{Gosson2006}
{\sc de~Gosson, M.}
\newblock {\em Symplectic Geometry and Quantum Mechanics}.
\newblock Birkhäuser Basel, 2006.

\bibitem{Gosson2011}
{\sc de~Gosson, M.~A.}
\newblock {\em {Symplectic Methods in Harmonic Analysis and in Mathematical
  Physics}}.
\newblock Springer Basel, 2011.

\bibitem{Feichtinger2006}
{\sc Feichtinger, H.~G.}
\newblock {Modulation Spaces: Looking Back and Ahead}.
\newblock {\em Sampl. Theory Signal Process. Data Anal. 5}, 2 (2006), 109--140.

\bibitem{Flandrin1999}
{\sc Flandrin, P.}
\newblock {\em Time-frequency/time-scale analysis}, vol.~10 of {\em Wavelet
  Analysis and its Applications}.
\newblock Academic Press, Inc., San Diego, CA, 1999.
\newblock With a preface by Yves Meyer, Translated from the French by Joachim
  St\"{o}ckler.

\bibitem{Folland1989}
{\sc Folland, G.~B.}
\newblock {\em {Harmonic Analysis in Phase Space. ({AM}-122)}}.
\newblock Princeton University Press, 1989.

\bibitem{FollandSitaram1997}
{\sc Folland, G.~B., and Sitaram, A.}
\newblock The uncertainty principle: a mathematical survey.
\newblock {\em J. Fourier Anal. Appl. 3}, 3 (1997), 207--238.

\bibitem{FuerRzeszotnik2018}
{\sc F\"{u}hr, H., and Rzeszotnik, Z.}
\newblock A note on factoring unitary matrices.
\newblock {\em Linear Algebra Appl. 547\/} (2018), 32--44.

\bibitem{gabor1946}
{\sc Gabor, D.}
\newblock {Theory of communication. Part 1: The analysis of information}.
\newblock {\em J. IEE (London), 93(III) 93}, 26 (1946), 429--441.

\bibitem{Groechenig2003}
{\sc Gr\"{o}chenig, K.}
\newblock Uncertainty principles for time-frequency representations.
\newblock In {\em Advances in {G}abor analysis}, Appl. Numer. Harmon. Anal.
  Birkh\"{a}user Boston, Boston, MA, 2003, pp.~11--30.

\bibitem{GroechenigZimmermann2001}
{\sc Gr\"{o}chenig, K., and Zimmermann, G.}
\newblock Hardy's theorem and the short-time {F}ourier transform of {S}chwartz
  functions.
\newblock {\em J. London Math. Soc. (2) 63}, 1 (2001), 205--214.

\bibitem{Groechenig2001}
{\sc Gröchenig, K.}
\newblock {\em {Foundations of Time-Frequency Analysis}}.
\newblock Birkhäuser Boston, 2001.

\bibitem{Hall2015}
{\sc Hall, B.}
\newblock {\em Lie groups, {L}ie algebras, and representations}, second~ed.,
  vol.~222 of {\em Graduate Texts in Mathematics}.
\newblock Springer International Publishing, 2015.

\bibitem{HavinJoericke1994}
{\sc  Havin, V.~P. and J{\"o}ricke, B.}
\newblock {\em The uncertainty principle in harmonic analysis}.
\newblock In {\em Commutative harmonic analysis, III}, pages 177--259,
  261--266. Springer, Berlin, 1995.

\bibitem{HlawatschBoudreaux1992}
{\sc Hlawatsch, F., and Boudreaux-Bartels, G.}
\newblock Linear and quadratic time-frequency signal representations.
\newblock {\em IEEE Signal Process. Mag. 9}, 2 (Apr. 1992), 21--67.

\bibitem{Iwasawa1949}
{\sc Iwasawa, K.}
\newblock On some types of topological groups.
\newblock {\em Ann. Math. 50}, 3 (1949), 507--558.

\bibitem{Jaming1998}
{\sc Jaming, P.}
\newblock Principe d'incertitude qualitatif et reconstruction de phase pour la
  transform\'{e}e de {W}igner.
\newblock {\em C. R. Acad. Sci. Paris S\'{e}r. I Math. 327}, 3 (1998),
  249--254.

\bibitem{Jaming2022}
{\sc Jaming, P.}
\newblock A simple observation on the uncertainty principle for the fractional
  {F}ourier transform.
\newblock {\em J. Fourier Anal. Appl. 28}, 3 (2022), Paper No. 51, 8.

\bibitem{janssentau}
{\sc Janssen, A.~J. E.~M.}
\newblock On the locus and spread of pseudodensity functions in the
  time-frequency plane.
\newblock {\em Philips J. Res. 37}, 3 (1982), 79--110.

\bibitem{Janssen1998_supp}
{\sc Janssen, A. J. E.~M.}
\newblock Proof of a conjecture on the supports of {W}igner distributions.
\newblock {\em J. Fourier Anal. Appl. 4}, 6 (1998), 723--726.

\bibitem{Knapp_LieBeyond}
{\sc Knapp, A.~W.}
\newblock {\em Lie groups beyond an introduction}, vol.~140 of {\em Progress in
  Mathematics}.
\newblock Birkh\"{a}user Boston, Inc., Boston, MA, 1996.

\bibitem{Namias1980}
{\sc Namias, V.}
\newblock The fractional order {F}ourier transform and its application to
  quantum mechanics.
\newblock {\em J. Inst. Math. Appl. 25}, 3 (1980), 241--265.

\bibitem{RicaudTorresani2014}
{\sc Ricaud, B., and Torr\'{e}sani, B.}
\newblock A survey of uncertainty principles and some signal processing
  applications.
\newblock {\em Adv. Comput. Math. 40}, 3 (2014), 629--650.



\bibitem{terMorscheOonincx2002O}
{\sc ter Morsche, H.~G., and Oonincx, P.~J.}
\newblock {On the Integral Representations for Metaplectic Operators}.
\newblock {\em J. Fourier Anal. Appl. 8\/} (2002), 245--258.

\bibitem{Torresani1999}
{\sc Torr\'{e}sani, B.}
\newblock Time-frequency analysis, from geometry to signal processing.
\newblock In {\em Contemporary problems in mathematical physics ({C}otonou,
  1999)}. World Sci. Publ., River Edge, NJ, 2000, pp.~74--96.

\bibitem{Weil1964}
{\sc Weil, A.}
\newblock Sur certains groupes d'opérateurs unitaires.
\newblock {\em Acta Math.}, 111 (1964), 143–211.

\bibitem{Wiener1929}
{\sc Wiener, N.}
\newblock {H}ermitian polynomials and {F}ourier analysis.
\newblock {\em J. Math. and Phys. 8}, 1–4 (Apr. 1929), 70--73.

\bibitem{Wilczok2000}
{\sc Wilczok, E.}
\newblock New uncertainty principles for the continuous {G}abor transform and
  the continuous wavelet transform.
\newblock {\em Doc. Math. 5\/} (2000), 201--226.

\bibitem{Woodward1964}
{\sc Woodward, P.~M.}
\newblock {\em Probability and information theory, with applications to radar},
  second~ed.
\newblock Pergamon Press, Oxford-Edinburgh-New York-Paris-Frankfurt, 1964.

\bibitem{Zhang2023}
{\sc Zhang, Z.}
\newblock Uncertainty principle for free metaplectic transformation.
\newblock {\em J. Fourier Anal. Appl. 29}, 6 (2023), Paper No. 71, 33.

\bibitem{Zhang2019}
{\sc Zhang, Z.-C.}
\newblock Uncertainty principle for linear canonical transform using matrix
  decomposition of absolute spread matrix.
\newblock {\em Digit. Signal Process. 89\/} (2019), 145--154.

\end{thebibliography}
\end{document}